\documentclass[12pt,a4paper]{elsarticle}
\usepackage[utf8]{inputenc}
\usepackage{amsmath}
\usepackage{natbib}
\usepackage{hyperref}
\usepackage{breqn}
\usepackage{amsthm}
\usepackage{amsfonts}
\usepackage{amssymb} 
\usepackage{graphicx}
\usepackage[left=2cm,right=2cm,top=2cm,bottom=2cm]{geometry}
\usepackage{lineno,hyperref}
\modulolinenumbers[5]

\newtheorem{theorem}{Theorem}[section]
\newtheorem{corollary}{Corollary}[section]
\newtheorem{lemma}{Lemma}[section]
\newtheorem{proposition}{Proposition}[section]
\newtheorem{conjecture}{Conjecture}[section]
\newtheorem*{remark}{Remark}
\newtheorem{example}{Example}[section]
\newtheorem*{definition}{Definition}

\DeclareUnicodeCharacter{2061}{}

\begin{document}
\begin{frontmatter}
\title{Repeated Integration and Explicit Formula for the $n$-th Integral of $x^m(\ln x)^{m'}$}
\author{Roudy El Haddad}
\address{Université La Sagesse, Faculté de génie, Polytech}
\ead{roudy1581999@live.com}

\begin{abstract}
Repeated integration is a major topic of integral calculus. In this article, we study repeated integration. In particular, we study repeated integrals and recurrent integrals. For each of these integrals, we develop reduction formulae for both the definite as well as indefinite form. These reduction formulae express these repetitive integrals in terms of single integrals.
We also derive a generalization of the fundamental theorem of calculus that expresses a definite integral in terms of an indefinite integral for repeated and recurrent integrals. 
From the recurrent integral formulae, we derive some partition identities. 
Then we provide an explicit formula for the $n$-th integral of $x^m(\ln x)^{m'}$ in terms of a shifted multiple harmonic star sum. Additionally, we use this integral to derive new expressions for the harmonic sum and repeated harmonic sum. 
\end{abstract}
\end{frontmatter}
\section{Introduction}
The author is extremely interested in the study of repetitive structures. In previous articles, the author studied three types of repetitive sums: recurrent sums \cite{RecurrentSums}, multiple sums \cite{PolynomialSums}, and repeated sums \cite{RepeatedSums}. In this article, we extend this study to integrals by considering two types of repetitive integrals: the repeated integral and the recurrent integral. Our aim is to obtain formulae to simplify these general integral structures in order to improve our ability to work with them. The development of integral calculus and multivariable calculus was revolutionary for all branches of science. Repeated integration is a fundamental tool used everywhere from mathematics to physics to biology to economics.
Due to this widespread use of such structures, improvements in the understanding of such structures could potentially have numerous applications. 
The first person to attempt simplifying repeated integrals was Cauchy \cite{Cauchy} who developed the following reduction formula: 
$$
\int_{a}^{x}{\cdots{\int_{a}^{x_{3}}{\int_{a}^{x_{2}}{f(x_1)\,dx_1\cdots dx_n}}}}
=\frac{1}{(n-1)!}{\int_{a}^{x}{(x-t)^{n-1}f(t)\,dt}}
.$$
This formula is generalized to non-integer parameters by the Riemann-Liouville integral \cite{riemann,liouville,AdvancedCal}. Both of these are generalized to arbitrary dimension by the Riesz potential \cite{riesz}. These formulae have a multitude of applications, in particular, they are used in fractional calculus in order to perform integrations or differentiations of fractional order \cite{FractionalCal}. 
In this article, we introduce a new reduction formula for both the definite and indefinite form of repeated integrals as well as we produce a generalization of the fundamental theorem of calculus linking the definite form to the indefinite form. We repeat this procedure for the second type of repetitive integrals studied, that is, recurrent integrals. In particular, the simplification formulas for definite recurrent integrals will involve partitions. Hence, exploiting this connection between recurrent integrals and partitions, we derive some partition identities. 

A particular repeated integral in which we are interested is the $n$-th integral of $x^m (\ln x)^{m'}$.  
We are interested in the repeated integral of $x^m (\ln x)^{m'}$ as developing a formula for such an integral will require the use of all three main theorems (the variation formula, inversion formula, and reduction formula) of \cite{RecurrentSums}. Hence, we present simplification formulae for this integral as an application to that cited article. This integral is surprisingly linked to the harmonic sum as well as to a more general form called the multiple harmonic star sum (which is a particular case of recurrent sums \cite{RecurrentSums}). In fact, the reduction formula for the repeated integral of $x^m (\ln x)^{m'}$ is expressed in terms of a recurrent sum (in particular, a shifted multiple harmonic star sum). The connection between repeated integrals and recurrent sum structures has been previously explored. In 2018, Ce Xu \cite{xu2018} illustrated the connection between similar integrals involving logarithms and a particular recurrent sum (the multiple zeta star function or multiple harmonic star sum). 
In this article, we further contribute to establishing this connection between logarithmic integrals and multiple harmonic star sums. 
The multiple harmonic star sum (MHSS) \cite{kuba2019,hoffman1996sums,murahara2019interpolation} is defined as follows: 
$$
\zeta_{n}^{\star}(s_1,s_2, \ldots, s_k)
=\sum_{1 \leq N_1 \leq N_2 \leq \cdots \leq N_k \leq n}{\frac{1}{N_1^{s_1}N_2^{s_2}\cdots N_k^{s_k}}}
.$$
For the purpose of this article, let us also define the following more general notation: 
$$
\zeta_{q,n}^{\star}(s_1,s_2, \ldots, s_k)
=\sum_{q \leq N_1 \leq N_2 \leq \cdots \leq N_k \leq n}{\frac{1}{N_1^{s_1}N_2^{s_2}\cdots N_k^{s_k}}}
.$$
Also, let $(s_1,\ldots,s_k)=(\{p\}_{k})$ represent $(s_1,\ldots,s_k)=(p, \ldots, p)$. 

The connection goes both ways: In the same way multiple harmonic star sums help in simplifying such integrals, these integrals can be used to simplify certain harmonic sums. Study \cite{RepeatedSums}, presents several formulae for simplifying harmonic sums, repeated harmonic sums as well as a variant involving binomial coefficients (that is referred to as the binomial-harmonic sum). Exploiting the connection established between logarithmic integrals and harmonic sums, these relations will be utilized to derive formulae for the repeated harmonic sum and the repeated binomial-harmonic sum. \\ 

In Section \ref{section 2}, we derive reduction formulas for indefinite repeated integrals as well as indefinite recurrent integrals which allow us to express such integrals in terms of simple (non-repetitive) integrals. Then, we present a generalization of the fundamental theorem of calculus for repeated integrals and recurrent integrals. These formulae will then be used to express the definite repeated and recurrent integrals in terms of definite single integrals. Partition identities will also be derived from the recurrent integral formulae. 
In Section \ref{section 3}, we propose explicit formulas for the $n$-th order repeated integral of $x^m(\ln x)^{m'}$. These formulas will include $n$, $m$, and $m'$ as parameters in the expression. They will also include certain number theoretical concepts such as MHSS and partitions. 
In Section \ref{section 4}, the relation between this integral and the harmonic sum will be exploited to produce results related to the harmonic sum such as an alternating sum representation of the harmonic sum. We also utilize the relations derived in \cite{RepeatedSums} to derive a similar alternating sum representation for the repeated harmonic sum as well as for the modified form of the repeated harmonic sum (the repeated binomial-harmonic sum). 
\section{Repeated integration} \label{section 2}
In this section, we consider two techniques of repeated integration: repeated integrals and recurrent integrals. We begin by defining these types of repetitive integrals and presenting some notation. 
\begin{definition}
We define the indefinite repeated integral  of order $n$ of the function $f(x)$, denoted $I_n[f(x)](x)$ or simply $I_n$, by the following structure: 
\begin{equation}
I_n=
\int{\cdots\int{f(x)\,dx\cdots dx}}
=\int{\cdots\int{f(x)\,dx^n}}
.\end{equation}
It can also be defined by the following recurrent relation: 
\begin{equation*}
\begin{cases}
I_{n+1}=\int{I_{n} \,dx}, \\
I_{0}\,\,\,\,\,\,=f(x). 
\end{cases}
\end{equation*}
\end{definition}
\begin{definition}
We define the $n$-th order definite repeated integral of the function $f(x)$, denoted $I_{n,a,b}[f(x)]$ or simply $I_{n,a,b}$, by the following structure: 
\begin{equation}
I_{n,a,b}=
\int_{a}^{b}{\cdots \int_{a}^{x_3}{\int_{a}^{x_2}{f(x_1)\,dx_1dx_2 \cdots dx_n}}}
.\end{equation}
\end{definition}
\begin{definition}
We define the indefinite recurrent integral of order $n$ of the function $f(x)$, denoted $J_n[f(x)](x)$ or simply $J_n$, by the following recurrent relation: 
\begin{equation*}
\begin{cases}
J_{n+1}=\int{f(x)J_{n} \,dx}, \\
J_{0}\,\,\,\,\,\,=1. 
\end{cases}
\end{equation*}
Explicitly, it can also be expressed as follows: 
\begin{equation}
J_n=
\int{f(x)\int{f(x)\cdots\int{f(x)\,dx^n}}}
.\end{equation}
\end{definition}
\begin{definition}
We define the $n$-th order definite recurrent integral of the function $f(x)$, denoted $J_{n,a,b}[f(x)]$ or simply $J_{n,a,b}$, by the following structure: 
\begin{equation}
J_{n,a,b}=
\int_{a}^{b}{f(x_n)\int_{a}^{x_n}{f(x_{n-1}) \cdots \int_{a}^{x_2}{f(x_1)\,dx_1 \cdots dx_n}}}
.\end{equation}
\end{definition}
\begin{remark}
{\em As we can see, repeated integrals are analogous to repeated sums \cite{RepeatedSums} while recurrent integrals are analogous to recurrent sums \cite{RecurrentSums}. }
\end{remark} 
Likewise, before we proceed, we need to generalize the concept of constant of integration: 
Let us define $C_{n}(x)$ as 
\begin{equation*}
C_{n}(x)=\sum_{i=0}^{n-1}{c_{i} \frac{x^i}{i!}}
\end{equation*}
where the $c_i$'s are constants of integration. 
\begin{remark}
{\em If we are using $C_{n}(x)$ as a function of $x$, for simplicity, we will denote it as $C_n$.}
\end{remark}
\subsection{Repeated integrals} 
In this section, we begin by providing a way to compute indefinite repeated integrals in terms of simple indefinite integrals. Then we generalize the concept of primitive and present a generalization of the fundamental theorem of calculus. Finally, we present a formula for computing definite repeated integrals in terms of simple definite integrals. 
\subsubsection{Indefinite repeated integrals}
We begin by proving a formula for expressing indefinite repeated integrals in terms of single integrals. We refer to this theorem as the reduction formula for indefinite repeated integrals. 
\begin{theorem} \label{t 2.1}
Let $f(x)$ be a function of $x$, for any $n\in\mathbb{N^*}$, we have that 
\begin{equation*}
\int{\cdots\int{f(x)dx^n}}
=\sum_{k=1}^{n}{(-1)^{n-k}\frac{x^{k-1}}{(k-1)!}\left(\int{\frac{x^{n-k}}{(n-k)!}f(x)dx}\right)}+C_n
.\end{equation*}
\end{theorem}
\begin{remark}
By substituting $k$ by $n-k$ and performing a few manipulations, this theorem can be rewritten as follows: 
\begin{equation*}
\int{\cdots\int{f(x)dx^n}}
=\frac{x^{n-1}}{(n-1)!}\sum_{k=0}^{n-1}{\binom{n-1}{k}\frac{(-1)^{k}}{x^{k}}\left(\int{x^{k}f(x)dx}\right)}+C_n
.\end{equation*}
\end{remark}
\begin{proof}
1. Base case: verify true for $n=1$. 
\begin{equation*}
\sum_{k=1}^{1}{(-1)^{1-k}\frac{x^{k-1}}{(k-1)!}\left(\int{\frac{x^{1-k}}{(1-k)!}f(x)dx}\right)}+C_1
=\int{f(x)\,dx}+C_1
.\end{equation*}
2. Induction hypothesis: assume the statement is true until $n$. 
\begin{equation*}
\int{\cdots\int{f(x)dx^n}}
=\sum_{k=1}^{n}{(-1)^{n-k}\frac{x^{k-1}}{(k-1)!}\left(\int{\frac{x^{n-k}}{(n-k)!}f(x)dx}\right)}+C_n
.\end{equation*}
3. Induction step: we will show that this statement is true for $(n+1)$. \\
We have to show the following statement to be true: 
\begin{equation*}
\int{\cdots\int{f(x)dx^{n+1}}}
=\sum_{k=1}^{n+1}{(-1)^{n-k+1}\frac{x^{k-1}}{(k-1)!}\left(\int{\frac{x^{n-k+1}}{(n-k+1)!}f(x)dx}\right)}+C_{n+1}
.\end{equation*}
$$ \\ $$
To be concise, we denote the left hand side term by $I$. By applying the induction hypothesis,  
\begin{equation*}
\begin{split}
I
&=\int{\left(\int{\cdots \int{f(x)\,dx^{n}}}\right)dx} \\
&=\int{\left(\sum_{k=1}^{n}{(-1)^{n-k}\frac{x^{k-1}}{(k-1)!}\left(\int{\frac{x^{n-k}}{(n-k)!}f(x)dx}\right)}+C_n\right)dx} \\
&=\int{\left(\sum_{k=1}^{n}{(-1)^{n-k}\frac{x^{k-1}}{(k-1)!}\left(\int{\frac{x^{n-k}}{(n-k)!}f(x)dx}\right)}\right)dx} +\int{\left(\sum_{i=0}^{n-1}{c_{i} \frac{x^i}{i!}}\right)\,dx}+C \\
&=\sum_{k=1}^{n}{(-1)^{n-k}\int{\frac{x^{k-1}}{(k-1)!}\left(\int{\frac{x^{n-k}}{(n-k)!}f(x)dx}\right)dx}} +\sum_{i=0}^{n-1}{c_{i} \frac{x^{i+1}}{(i+1)!}}+C 
.\end{split}
\end{equation*}
Setting $du=\frac{x^{k-1}}{(k-1)!}dx$ and $v=\int{\frac{x^{n-k}}{(n-k)!}f(x)\,dx}$, then applying integration by parts, we have 
\begin{equation*}
\begin{split}
\int{\frac{x^{k-1}}{(k-1)!}\left(\int{\frac{x^{n-k}}{(n-k)!}f(x)dx}\right)dx}
&=\frac{x^k}{k!}\int{\frac{x^{n-k}}{(n-k)!}f(x)\,dx}
-\int{\frac{x^k}{k!}\frac{x^{n-k}}{(n-k)!}f(x)\,dx}\\
&=\frac{x^k}{k!}\int{\frac{x^{n-k}}{(n-k)!}f(x)\,dx}
-\binom{n}{k}\int{\frac{x^{n}}{n!}f(x)\,dx}
.\end{split}
\end{equation*} 
Substituting back, we get 
\begin{equation*}
\begin{split}
I
&=\sum_{k=1}^{n}{(-1)^{n-k}\frac{x^k}{k!}\int{\frac{x^{n-k}}{(n-k)!}f(x)\,dx}}
-\sum_{k=1}^{n}{(-1)^{n-k}\binom{n}{k}\int{\frac{x^{n}}{n!}f(x)\,dx}} 
+\sum_{i=0}^{n-1}{c_{i} \frac{x^{i+1}}{(i+1)!}}+C \\
&\resizebox{\linewidth-7pt}{!}{$\displaystyle{=\sum_{k=1}^{n}{(-1)^{n-k}\frac{x^k}{k!}\int{\frac{x^{n-k}}{(n-k)!}f(x)\,dx}}
-(-1)^{n}\left(\int{\frac{x^{n}}{n!}f(x)\,dx}\right)\sum_{k=1}^{n}{(-1)^{k}\binom{n}{k}} 
+\sum_{i=1}^{n}{c_{i-1} \frac{x^{i}}{i!}}+C. }$}
\end{split}
\end{equation*} 
Knowing that $\sum_{k=0}^{n}{(-1)^{k}\binom{n}{k}}=0$, 
$$
\sum_{k=1}^{n}{(-1)^{k}\binom{n}{k}} =\sum_{k=0}^{n}{(-1)^{k}\binom{n}{k}} -1 
=-1.
$$
Also, let us consider the set of constants $b_i$ such that $b_0=C$ and, for $i \geq 1$, $b_i=c_{i-1}$. Hence, 
\begin{equation*}
\begin{split}
I
&=\sum_{k=1}^{n}{(-1)^{n-k}\frac{x^k}{k!}\int{\frac{x^{n-k}}{(n-k)!}f(x)\,dx}}
+(-1)^{n}\left(\int{\frac{x^{n}}{n!}f(x)\,dx}\right)
+\sum_{i=0}^{n}{b_{i} \frac{x^{i}}{i!}} \\
&=\sum_{k=2}^{n+1}{(-1)^{n-k+1}\frac{x^{k-1}}{(k-1)!}\int{\frac{x^{n-k+1}}{(n-k+1)!}f(x)\,dx}}
+(-1)^{n}\left(\int{\frac{x^{n}}{n!}f(x)\,dx}\right)
+\sum_{i=0}^{n}{b_{i} \frac{x^{i}}{i!}} \\
&=\sum_{k=1}^{n+1}{(-1)^{n-k+1}\frac{x^{k-1}}{(k-1)!}\int{\frac{x^{n-k+1}}{(n-k+1)!}f(x)\,dx}}
+\sum_{i=0}^{n}{b_{i} \frac{x^{i}}{i!}}
.\end{split}
\end{equation*} 
Noting that $\sum_{i=0}^{n}{b_{i} \frac{x^{i}}{i!}}$ represents $C_{n+1}$, the theorem is proven by induction. 
\end{proof}
Hence, to find the $n$-th integral of $f(x)$ according to Theorem \ref{t 2.1}, it is enough to find an expression for the integral of $x^k f(x)$ in terms of $k$. Thus, the problem is reduced from computing an $n$-th integral to computing a simple integral.  
Using this theorem, we find two propositions that will be useful later on. 
\begin{proposition} \label{repeated integral of 1}
{\em For $f(x)=1$, using Theorem \ref{t 2.1}, we get that } 
\begin{equation*} 
\int{\cdots \int{1}\,dx^n}=\frac{x^n}{n!}+C_n
.\end{equation*}
\end{proposition}
\begin{proof}
Applying Theorem \ref{t 2.1} for $f(x)=1$, we get
\begin{equation*}
\begin{split}
\int{\cdots\int{1\,dx^n}}
&=\frac{x^{n-1}}{(n-1)!}\sum_{k=0}^{n-1}{\binom{n-1}{k}\frac{(-1)^{k}}{x^{k}}\left(\int{x^{k}dx}\right)}+C_n \\
&=\frac{x^{n}}{(n-1)!}\sum_{k=0}^{n-1}{\binom{n-1}{k}\frac{(-1)^{k}}{k+1}}+C_n \\
&=\frac{x^{n}}{n!}\sum_{k=0}^{n-1}{\binom{n}{k+1}(-1)^{k}}+C_n \\
&=\frac{x^{n}}{n!}+C_n
.\end{split}
\end{equation*}
\end{proof}
\begin{proposition} \label{repeated integral of ln}
{\em For $f(x)=\ln x$, using Theorem \ref{t 2.1}, we get that } 
\begin{equation*} 
\int{\cdots \int{\ln x}\,dx^n}-C_{n}
=\int{\cdots \int{\frac{1}{x}}\,dx^{n+1}}-C_{n+1}
=\frac{x^{n}}{n!}\left[ \ln x +\sum_{k=1}^{n}{\binom{n}{k}\frac{(-1)^k}{k}} \right]
.\end{equation*}
\end{proposition}
\begin{remark}
{\em Theorem \ref{t 2.1} can also be generalized to non-integer orders to obtain: }
\begin{equation*}
I_{\alpha}
=\frac{x^{\alpha-1}}{\Gamma(\alpha)}\sum_{k=0}^{\infty}{\binom{\alpha-1}{k}\frac{(-1)^k}{x^k}\int{x^k f(x)\,dx}}
\end{equation*}
{\em where} 
$$
\binom{\alpha-1}{k}=\frac{(\alpha-1)\cdots(\alpha-k)}{k!}.
$$
\end{remark}
\subsubsection{Relation between definite and indefinite repeated integrals}
The fundamental theorem of calculus relates the definite integral of a function to its primitive. In this section, we develop a generalization of the fundamental theorem of calculus that relates the definite repeated integral of a function to its repeated primitives. 

We begin by defining and generalizing the concept of primitive. 
\begin{definition}
We denote by $F(x)$ the primitive of $f(x)$. We use primitive to mean the expression of the indefinite integral without the constant of integration (In other words, the case where the constant of integration is zero). 
\end{definition}
\begin{definition}
We define the primitive of order $n$ of the function $f(x)$, denoted $F_n[f(x)](x)$ or $F_n(x)$, as the expression of the $n$-th indefinite integral of $f(x)$ excluding the constants of integration (In other words, the case where the constants of integration are zero). 
\end{definition}
The following theorem illustrates how to compute the $n$-th primitive of a function $f(x)$ in terms of a simple primitive, specifically, the primitive of $x^k f(x)$. 
\begin{theorem} \label{t 2.3}
Let $\Phi_{k}(x)$ be the primitive of $x^{k}f(x)$. The $n$-th primitive of $f(x)$ can be expressed as follows: 
\begin{equation*}
F_n(x)
=\sum_{k=1}^{n}{(-1)^{n-k}\frac{x^{k-1}}{(k-1)!}\frac{\Phi_{n-k}(x)}{(n-k)!}}
=\frac{x^{n-1}}{(n-1)!}\sum_{k=0}^{n-1}{\binom{n-1}{k}\frac{\Phi_{k}(x)}{(-x)^{k}}}
.\end{equation*}
\end{theorem}
\begin{proof}
Using the definition of the $n$-th primitive presented as well as the notation presented, we can see, from Theorem \ref{t 2.1}, that the $n$-th primitive of $f(x)$ can be expressed as stated by this theorem. 
\end{proof}
\begin{example}
The $n$-th primitive of $f(x)=1$ is 
\begin{equation}
F_n(x)=\frac{x^n}{n!}.
\end{equation}
\end{example}
\begin{example}
The $n$-th primitive of $f(x)=\ln x$ is 
\begin{equation}
F_n[\ln x](x)=F_{n+1}\left[\frac{1}{x}\right](x)
=\frac{x^{n}}{n!}\left[ \ln x +\sum_{k=1}^{n}{\binom{n}{k}\frac{(-1)^k}{k}} \right].
\end{equation}
\end{example}
In the following lemma, we present a modified version of the fundamental theorem of calculus which is needed to derive a generalization of the fundamental theorem of calculus for repeated integrals. 
\begin{lemma} \label{l 2.1}
We have that 
\begin{equation*}
\int_{a}^{b}{F_{n}(x)dx}
=F_{n+1}(x)|_{a}^{b}=F_{n+1}(b)-F_{n+1}(a)
.\end{equation*}
\end{lemma}
\begin{proof}
$F_{n+1}(x)$ is the primitive of $F_{n}(x)$. Hence, from the fundamental theorem of calculus, we have the lemma. 
\end{proof}
In what follows, we prove a generalization of the fundamental theorem of calculus. 
\begin{theorem} \label{t 2.2}
Let $f(x)$ be a function defined in the interval $[a,b]$, for any $n\in\mathbb{N^*}$, we have
\begin{equation*}
\int_{a}^{b}{\cdots \int_{a}^{x_3}{\int_{a}^{x_2}{f(x_1)dx_1dx_2 \cdots dx_n}}}
=F_n(b)-\sum_{k=1}^{n}{\frac{(b-a)^{n-k}}{(n-k)!}F_k(a)}
.\end{equation*}
\end{theorem}
\begin{remark}
{\em By applying the binomial theorem, this theorem can be expressed as follows:  }
\begin{equation*}
\int_{a}^{b}{\cdots \int_{a}^{x_3}{\int_{a}^{x_2}{f(x_1)dx_1dx_2 \cdots dx_n}}}
=F_n(b)-\sum_{k=1}^{n}{F_k(a)\frac{b^{n-k}}{(n-k)!}\sum_{\ell=0}^{n-k}{\binom{n-k}{\ell}(-1)^{\ell}\left(\frac{a}{b}\right)^{\ell}}}
.\end{equation*}
\end{remark}
\begin{proof}
1. Base case: verify true for $n=1$. 
$$
F_1(b)-\sum_{k=1}^{1}{\frac{(b-a)^{1-k}}{(1-k)!}F_k(a)}
=F_1(b)-F_1(a)
=\int_{a}^{b}{f(x_1)dx_1}
.$$
2. Induction hypothesis: assume the statement is true until $n$.
\begin{equation*}
\int_{a}^{b}{\cdots \int_{a}^{x_3}{\int_{a}^{x_2}{f(x_1)dx_1dx_2 \cdots dx_n}}}
=F_n(b)-\sum_{k=1}^{n}{\frac{(b-a)^{n-k}}{(n-k)!}F_k(a)}
.\end{equation*}
3. Induction step: we will show that this statement is true for $(n+1)$. \\
We have to show the following statement to be true:
\begin{equation*}
\int_{a}^{b}{\cdots \int_{a}^{x_3}{\int_{a}^{x_2}{f(x_1)dx_1dx_2 \cdots dx_{n+1}}}}
=F_{n+1}(b)-\sum_{k=1}^{n+1}{\frac{(b-a)^{n-k+1}}{(n-k+1)!}F_k(a)}
.\end{equation*} 
$$ \\ $$  
\begin{equation*}
\int_{a}^{b}{\cdots \int_{a}^{x_3}{\int_{a}^{x_2}{f(x_1)dx_1dx_2 \cdots dx_{n+1}}}}
=\int_{a}^{b}{\left[\int_{a}^{x_{n+1}}{\cdots \int_{a}^{x_3}{\int_{a}^{x_2}{f(x_1)dx_1dx_2 \cdots dx_n}}}\right]dx_{n+1}}
.\end{equation*} 
Letting $t=x_{n+1}$ and applying the induction hypothesis, 
\begin{equation*}
\begin{split}
\int_{a}^{b}{\cdots \int_{a}^{x_3}{\int_{a}^{x_2}{f(x_1)dx_1dx_2 \cdots dx_{n+1}}}}
&=\int_{a}^{b}{\left[F_n(x_{n+1})-\sum_{k=1}^{n}{\frac{(x_{n+1}-a)^{n-k}}{(n-k)!}F_k(a)}\right]dx_{n+1}} \\
&=\int_{a}^{b}{F_n(t)dt}
-
\sum_{k=1}^{n}{F_k(a)\int_{a}^{b}{\frac{(t-a)^{n-k}}{(n-k)!}dt}}
.\end{split}
\end{equation*}
Let us note that, letting $u=(t-a)$ and, hence, $du=dt$, we have 
$$
\int_{a}^{b}{\frac{(t-a)^{n-k}}{(n-k)!}dt}
=\int_{a}^{b}{\frac{u^{n-k}}{(n-k)!}du}
=\left[ \frac{u^{n-k+1}}{(n-k+1)!} \right]_{a}^{b}
=\frac{(b-a)^{n-k+1}}{(n-k+1)!}
.$$
Using Lemma 2.1, we get 
\begin{equation*}
\begin{split}
\int_{a}^{b}{\cdots \int_{a}^{x_3}{\int_{a}^{x_2}{f(x_1)dx_1dx_2 \cdots dx_{n+1}}}}
&=F_{n+1}(b)-F_{n+1}(a)
-\sum_{k=1}^{n}{\frac{(b-a)^{n-k+1}}{(n-k+1)!}F_k(a)} \\
&=F_{n+1}(b)
-\sum_{k=1}^{n+1}{\frac{(b-a)^{n-k+1}}{(n-k+1)!}F_k(a)}
.\end{split}
\end{equation*}
The case for $(n+1)$ is proven. Hence, the theorem is proven by induction. 
\end{proof}
\begin{proposition} \label{definite repeated integral of 1}
{\em For $f(x)=1$, from Theorem \ref{t 2.2}, we have that }
\begin{equation*}
\int_{a}^{b}{\cdots \int_{a}^{x_3}{\int_{a}^{x_2}{1 \, dx_1dx_2 \cdots dx_n}}}
=\frac{(b-a)^n}{n!}.
\end{equation*}
\end{proposition}
\begin{proof}
By applying Theorem \ref{t 2.2} for $f(x)=1$, we have that
\begin{equation*}
\begin{split}
\int_{a}^{b}{\cdots \int_{a}^{x_3}{\int_{a}^{x_2}{1 \, dx_1dx_2 \cdots dx_n}}}
&=\frac{b^n}{n!}-\sum_{k=1}^{n}{\frac{(b-a)^{n-k}}{(n-k)!}\frac{a^k}{k!}}\\
&=\frac{b^n}{n!}-\frac{1}{n!}\sum_{k=1}^{n}{\binom{n}{k}(b-a)^{n-k}a^k}\\
&=\frac{b^n}{n!}-\left[\frac{(b-a+a)^n}{n!}-\frac{(b-a)^n}{n!}\right]\\
&=\frac{(b-a)^n}{n!}
.\end{split}
\end{equation*}
\end{proof}
\subsubsection{Definite repeated integrals}
In this section, we develop a reduction formula for definite repeated integrals which expresses such integrals in terms of definite single integrals. According to Theorem \ref{definite repeated integral}, it is enough to determine an expression for the definite integral of $x^k f(x)$ in function of $k$ to compute the $n$-th definite integral of $f(x)$. Hence, the problem is reduced from computing an $n$-th integral to computing a single integral. 
\begin{theorem} \label{definite repeated integral}
Let $f(x)$ be a function defined in the interval $[a,b]$, for any $n\in \mathbb{N^*}$, we have
\begin{equation*}
\int_{a}^{b}{\cdots \int_{a}^{x_3}{\int_{a}^{x_2}{f(x_1)dx_1dx_2 \cdots dx_n}}}
=\frac{b^{n-1}}{(n-1)!}\sum_{k=0}^{n-1}{\binom{n-1}{k}\frac{(-1)^k}{b^k}\int_{a}^{b}{x^k f(x)\,dx}}
.\end{equation*}
\end{theorem}
\begin{proof}
1. Base case: verify true for $n=1$. 
\begin{equation*}
\frac{b^{1-1}}{(1-1)!}\sum_{k=0}^{1-1}{\binom{1-1}{k}\frac{(-1)^k}{b^k}\int_{a}^{b}{x^k f(x)\,dx}}
=\int_{a}^{b}{ f(x)\,dx}
.\end{equation*}
2. Induction hypothesis: assume the statement is true until $n$.
\begin{equation*}
\int_{a}^{b}{\cdots \int_{a}^{x_3}{\int_{a}^{x_2}{f(x_1)dx_1dx_2 \cdots dx_n}}}
=\frac{b^{n-1}}{(n-1)!}\sum_{k=0}^{n-1}{\binom{n-1}{k}\frac{(-1)^k}{b^k}\int_{a}^{b}{x^k f(x)\,dx}}
.\end{equation*}
3. Induction step: we will show that this statement is true for $(n+1)$. \\
We have to show the following statement to be true:
\begin{equation*}
\int_{a}^{b}{\cdots \int_{a}^{x_3}{\int_{a}^{x_2}{f(x_1)dx_1dx_2 \cdots dx_{n+1}}}}
=\frac{b^{n}}{n!}\sum_{k=0}^{n}{\binom{n}{k}\frac{(-1)^k}{b^k}\int_{a}^{b}{x^k f(x)\,dx}}
.\end{equation*}
\\
To be concise, we denote the left hand side term by $I$. By applying the induction hypothesis, 
\begin{equation*}
\begin{split}
I
&=\int_{a}^{b}{\left[\int_{a}^{x_{n+1}}{\cdots \int_{a}^{x_3}{\int_{a}^{x_2}{f(x_1)dx_1dx_2 \cdots dx_n}}}\right]dx_{n+1}} \\ 
&=\int_{a}^{b}{\left[\frac{(x_{n+1})^{n-1}}{(n-1)!}\sum_{k=0}^{n-1}{\binom{n-1}{k}\frac{(-1)^k}{(x_{n+1})^k}\int_{a}^{x_{n+1}}{x^k f(x)\,dx}}\right]dx_{n+1}} \\ 
&=\frac{1}{(n-1)!}\sum_{k=0}^{n-1}{\binom{n-1}{k}(-1)^k\int_{a}^{b}{(x_{n+1})^{n-k-1}\int_{a}^{x_{n+1}}{x^k f(x)\,dx}dx_{n+1}}} 
.\end{split}
\end{equation*}
Let $t=x_{n+1}$. Setting $du=t^{n-k-1}dt$ and $v=\int_{a}^{t}{x^k f(x)\,dx}$, then applying integration by parts, we have 
\begin{equation*}
\begin{split}
\int_{a}^{b}{vdu} 
=[uv]_{a}^{b}-\int_{a}^{b}{udv} 
&=\left[\frac{t^{n-k}}{n-k}\int_{a}^{t}{x^k f(x)\,dx}\right]_{t=a}^{t=b}-\int_{a}^{b}{\frac{t^{n-k}}{n-k}[t^k f(t)]\,dt}\\
&=\frac{b^{n-k}}{n-k}\int_{a}^{b}{x^k f(x)\,dx}
-\int_{a}^{b}{\frac{t^{n}}{n-k} f(t)\,dt}
.\end{split}
\end{equation*}
Hence, substituting back, we get 
\begin{equation*}
\begin{split}
I
&\resizebox{\linewidth-15pt}{!}{$\displaystyle{=\frac{1}{(n-1)!}\sum_{k=0}^{n-1}{\binom{n-1}{k}(-1)^k\frac{b^{n-k}}{n-k}\int_{a}^{b}{x^k f(x)\,dx}} 
-\frac{1}{(n-1)!}\sum_{k=0}^{n-1}{\binom{n-1}{k}(-1)^k\int_{a}^{b}{\frac{t^{n}}{n-k} f(t)\,dt}}}$} \\
&=\frac{b^n}{n!}\sum_{k=0}^{n-1}{\binom{n}{k}\frac{(-1)^k}{b^k}\left(\int_{a}^{b}{x^k f(x)\, dx}\right)}
-\left(\int_{a}^{b}{\frac{t^n}{n!} f(t)\, dt}\right) \sum_{k=0}^{n-1}{\binom{n}{k} (-1)^k}
.\end{split}
\end{equation*}
Noticing that 
$$
\sum_{k=0}^{n-1}{\binom{n}{k} (-1)^k}=\sum_{k=0}^{n}{\binom{n}{k} (-1)^k}-(-1)^n
=-(-1)^n
$$
and that 
$$
\frac{b^n}{n!}\sum_{k=n}^{n}{\binom{n}{k}\frac{(-1)^k}{b^k}\left(\int_{a}^{b}{x^k f(x)\, dx}\right)}
=(-1)^n \left(\int_{a}^{b}{\frac{t^n}{n!} f(t)\, dt}\right), 
$$
then substituting back, we obtain the theorem. 
\end{proof}
\begin{remark}
Using Theorem \ref{definite repeated integral}, we can provide an additional proof for Proposition \ref{definite repeated integral of 1}:  
\begin{equation*}
\begin{split}
\int_{a}^{b}{\cdots \int_{a}^{x_3}{\int_{a}^{x_2}{1\,dx_1dx_2 \cdots dx_n}}}
&=\frac{b^{n-1}}{n!}\sum_{k=0}^{n-1}{\binom{n}{k+1}\frac{(-1)^k}{b^k}[b^{k+1}-a^{k+1}]} \\
&=\frac{b^{n}}{n!}\left[\sum_{k=1}^{n}{\binom{n}{k}(-1)^{k+1}}-\sum_{k=1}^{n}{\binom{n}{k}(-1)^{k+1}\left(\frac{a}{b}\right)^{k}}\right] \\
&=\frac{1}{n!}\sum_{k=0}^{n}{\binom{n}{k}(-1)^{k}b^{n-k}a^{k}} \\ 
&=\frac{(b-a)^n}{n!}.
\end{split}
\end{equation*}
\end{remark}
\begin{remark}
{\em Theorem \ref{definite repeated integral} can also be generalized to non-integer orders to obtain: }
\begin{equation*}
I_{\alpha,a,b}
=\frac{b^{\alpha-1}}{\Gamma(\alpha)}\sum_{k=0}^{\infty}{\binom{\alpha-1}{k}\frac{(-1)^k}{b^k}\int_{a}^{b}{x^k f(x)\,dx}}
\end{equation*}
{\em where} 
$$
\binom{\alpha-1}{k}=\frac{(\alpha-1)\cdots(\alpha-k)}{k!}.
$$
\end{remark}
\subsection{Recurrent integrals}
In this section, we present formulas for computing indefinite and definite recurrent integrals of a function in terms of the primitive of this same function. We also introduce the concept of recurrent primitive and present a generalization of the fundamental theorem of calculus for recurrent integrals. We finish by deriving a few partition identities from this type of integrals. 
\subsubsection{Indefinite recurrent integrals}
Now we prove a formula for expressing indefinite recurrent integrals in terms of single integrals. We can call this the reduction formula for indefinite recurrent integrals. 
\begin{theorem} \label{t 2.4}
Let $f(x)$ be a function of $x$, for any $n\in\mathbb{N^*}$, we have that 
\begin{equation*}
\int{f(x) \cdots \int{f(x)\int{f(x)\,dx^n}}}
=\frac{1}{n!}\left(\int{f(x)dx}\right)^n
+\sum_{i=0}^{n-1}{c_i\frac{1}{i!}\left(\int{f(x)dx}\right)^i}
.\end{equation*} 
\end{theorem} 
\begin{remark}
{\em Using the notation introduced, this theorem can be rewritten as }
\begin{equation*}
\int{f(x) \cdots \int{f(x)\int{f(x)\,dx^n}}}
=\frac{1}{n!}\left(F(x)\right)^n
+C_n(F(x))
.\end{equation*}
\end{remark}
\begin{proof}
Let $X=\int{f(x)\,dx}$ and $dX=f(x)\,dx$, we can rewrite the recurrent integral as follows: 
\begin{equation*}
\int{f(x) \cdots \int{f(x)\int{f(x)\,dx^n}}}
=\int{ \cdots \int{\int{1\,dX^n}}}.
\end{equation*}
Using Proposition \ref{repeated integral of 1}, we get 
\begin{equation*}
\int{f(x) \cdots \int{f(x)\int{f(x)\,dx^n}}}
=\frac{X^n}{n!}+C_{n}(X)
=\frac{1}{n!}\left(\int{f(x)dx}\right)^n
+\sum_{i=0}^{n-1}{c_i\frac{1}{i!}\left(\int{f(x)dx}\right)^i}.
\end{equation*}
\end{proof}
\begin{example} 
{\em For the recurrent integral of order $3$ of $\sin x$, we have }
\begin{equation*}
\begin{split}
\int{\sin x \int{\sin x \int{\sin x \, dx}dx}dx}
&=\int{\sin x \int{\sin x \left(-\cos x +c_2\right)\,dx}dx} \\
&=\int{\sin x \left(\frac{\cos^2 x}{2}-c_2\cos x +c_1\right)\,dx} \\
&=-\frac{\cos^3 x}{6}+c_2\frac{\cos^2 x}{2} -c_1 \cos x+c_0 \\
&=\frac{1}{3!}\left(\int{\sin x \,dx}\right)^3
+\sum_{i=0}^{2}{c_i\frac{1}{i!}\left(\int{\sin x \,dx}\right)^i}
.\end{split}
\end{equation*}
\end{example}
\begin{example} 
{\em For the recurrent integral of order $3$ of $e^x$, we have }
\begin{equation*}
\begin{split}
\int{e^x \int{e^x \int{e^x \, dx}dx}dx}
&=\int{e^x \int{e^x \left(e^x +c_2\right)\,dx}dx} \\
&=\int{e^x \left(\frac{e^{2x}}{2}+c_2e^x +c_1\right)\,dx} \\
&=\frac{e^{3x}}{6}+c_2\frac{e^{2x}}{2}+c_1 e^x+c_0 \\
&=\frac{1}{3!}\left(\int{e^x \,dx}\right)^3
+\sum_{i=0}^{2}{c_i\frac{1}{i!}\left(\int{e^x \,dx}\right)^i}
.\end{split}
\end{equation*}
\end{example}
\subsubsection{Relation between definite and indefinite recurrent integrals}
We begin by introducing the concept of recurrent primitive. 
\begin{definition}
We define the recurrent primitive of order $n$ of the function $f(x)$, denoted $\rho_n(x)$, as the resulting expression of the $n$-th indefinite recurrent integral of $f(x)$ excluding the constants of integration (In other words, the case where the constants of integration are zero). 
\end{definition}
\begin{theorem} \label{t 2.6}
The $n$-th recurrent primitive of $f(x)$ can be expressed as follows: 
\begin{equation*}
\rho_{n}(x)
=\frac{1}{n!}\left(\rho_{1}(x)\right)^n
=\frac{1}{n!}\left(F(x)\right)^n
.\end{equation*} 
\end{theorem} 
\begin{proof}
Using the definition of the $n$-th recurrent primitive presented as well as the notation presented, we can see, from Theorem \ref{t 2.4}, that the $n$-th recurrent primitive of $f(x)$ can be expressed as follows: 
\begin{equation*}
\rho_{n}(x)
=\frac{1}{n!}\left(F(x)\right)^n
.\end{equation*}
Noticing that $\rho_1(x)$ is simply $F(x)$, the proof is complete. 
\end{proof}
\noindent This definition of $\rho_n(x)$ will be used in the next section to derive a reduction formula for the definite recurrent integral of a given function $f(x)$. 

Now we prove a variant of the fundamental theorem of calculus for recurrent integration. This is needed to prove an expression for definite recurrent integrals in terms of indefinite recurrent integrals. 
\begin{lemma} \label{l 2.3}
We have that 
\begin{equation*}
\int_{a}^{b}{f(x)\rho_{n}(x)\,dx}
=\rho_{n+1}(b)-\rho_{n+1}(a)
.\end{equation*}
\end{lemma} 
\begin{proof}
By definition of a recurrent integral, $\rho_{n+1}(x)=\int{f(x)\rho_{n}(x) \,dx}$. Hence, $\rho_{n+1}(x)$ is the primitive of $f(x)\rho_{n}(x)$. Therefore, by the fundamental theorem of calculus, we obtain the lemma. 
\end{proof}
Before we can proceed, we need to present the concept of partition of an integer as partitions are involved in the generalized fundamental theorem of calculus for recurrent integrals. As defined by the author in \cite{RecurrentSums,PolynomialSums}, a partition is defined as follows: 
\begin{definition}
A partition of a non-negative integer $m$ is a set of positive integers whose sum equals $m$.
We can represent a partition of $m$ as an ordered set $(y_{k,1},\ldots,y_{k,m})$ that verifies
\begin{equation}
y_{k,1}+2y_{k,2}+ \cdots + my_{k,m}
=\sum_{i=1}^{m}{i\,y_{k,i}}
=m.
\end{equation}
\end{definition} 
\noindent The coefficient $y_{k,i}$ is the multiplicity of the integer $i$ in the $k$-th partition of $m$. Note that $0\leq y_{k,i}\leq m$ while $1\leq i \leq m$. Also note that the number of partitions of an integer $m$ is given by the partition function denoted $p(n)$. Another way of representing a partition of an integer is with a set $(\lambda_1, \ldots, \lambda_{\ell})$ that satisfies: 
\begin{equation}
\begin{cases}
\lambda_1+\cdots+\lambda_{\ell}=\sum_{i=1}^{\ell}{\lambda_i}=m, \\
\lambda_i\leq\lambda_{i+1}.
\end{cases}
\end{equation}
Note that the second condition is necessary to ensure a unique representation for each partition of the integer. The $\lambda_i$'s are called the parts of this partition and $\ell$ is called the length of this given partition. Note that $1\leq \lambda_i\leq m$ and that the $\lambda_i$'s are not necessarily all distinct. Also note that the partitions of an integer do not all have the same length. In terms of the first definition, $\ell$ is the sum of the multiplicities: $\ell=\sum_{i=1}^{m}{y_i}=y_1+\cdots+y_{m}$. 
As partitions are not a fundamental part of this article, we will not go into more details. However, for readers interested in a detailed explanation of partitions, see \cite{andrews1998theory}. 

Now that the concept of partition of an integer has been presented, we introduce the generalized fundamental theorem of calculus for recurrent integrals. 
\begin{theorem} \label{t 2.5}
Let $f(x)$ be a function defined in the interval $[a,b]$, for any $n\in\mathbb{N^*}$, we have
\begin{equation*} 
\begin{split}
&\int_{a}^{b}{f(x_n)\int_{a}^{x_n}{f(x_{n-1}) \cdots \int_{a}^{x_2}{f(x_1)\,dx_1 \cdots dx_n}}} \\
&\,\,\,\,=\sum_{k=1}^{n}{(\rho_{k}(b)-\rho_{k}(a))\sum_{\sum{iy_i}=n-k}}{\,\,\prod_{i=1}^{n-k}{(-1)^{y_i}(\rho_{i}(a))^{y_i}}}
.\end{split}
\end{equation*} 
\end{theorem} 
\begin{remark} 
{\em This theorem can also be rewritten using the second partition notation: }
\begin{equation*} 
\begin{split}
\int_{a}^{b}{f(x_n)\int_{a}^{x_n}{f(x_{n-1}) \cdots \int_{a}^{x_2}{f(x_1)\,dx_1 \cdots dx_n}}}
=\sum_{k=1}^{n}{\sum_{\substack{\sum{\lambda_i}=n-k  \\  \lambda_i \leq \lambda_{i+1}}}{(\rho_{k}(b)-\rho_{k}(a))(-1)^{\ell}\prod_{i=1}^{\ell}{\rho_{\lambda_i}(a)}}}
.\end{split}
\end{equation*} 
{\em This sum is over all partitions $(\lambda_1, \ldots, \lambda_{\ell})$ of all integers $n-k$ with $1\leq k \leq n$. It is important to note that it is not equivalent to a sum over all partitions $(k,\lambda_1, \ldots, \lambda_{\ell})$ of the integer $n$ as $k$ does not necessarily satisfy $k\leq \lambda_1$. }
\end{remark}
\begin{proof}
1. Base case: verify true for $n=1$
\begin{equation*} 
\begin{split}
\sum_{k=1}^{1}{(\rho_{k}(b)-\rho_{k}(a))\sum_{\sum{iy_i}=1-k}}{\,\,\prod_{i=1}^{1-k}{(-1)^{y_i}(\rho_{i}(a))^{y_i}}}
=\rho_1(b)-\rho_1(a)
.\end{split}
\end{equation*} 
Likewise, from Lemma \ref{t 2.3}, 
$$
\int_{a}^{b}{f(x_1)\,dx_1 } =\int_{a}^{b}{f(x_1)\rho_0(x_1)\,dx_1 }=\rho_1(b)-\rho_1(a)
.$$
2. Induction hypothesis: assume the statement is true until $n$. 
\begin{equation*} 
\begin{split}
J_{n,a,b}
=\sum_{k=1}^{n}{(\rho_{k}(b)-\rho_{k}(a))\sum_{\sum{iy_i}=n-k}}{\,\,\prod_{i=1}^{n-k}{(-1)^{y_i}(\rho_{i}(a))^{y_i}}}
.\end{split}
\end{equation*} 
3. Induction step: we will show that this statement is true for $(n+1)$. \\
We have to show the following statement to be true: 
\begin{equation*} 
\begin{split}
J_{n+1,a,b}
=\sum_{k=1}^{n+1}{(\rho_{k}(b)-\rho_{k}(a))\sum_{\sum{iy_i}=n-k+1}}{\,\,\prod_{i=1}^{n-k+1}{(-1)^{y_i}(\rho_{i}(a))^{y_i}}}
.\end{split}
\end{equation*} 
$$ \\ $$ 
Using the induction hypothesis, 
\begin{equation*} 
\begin{split}
J_{n+1,a,b}
&=\int_{a}^{b}{f(x_{n+1})\left(\int_{a}^{x_{n+1}}{f(x_{n}) \cdots \int_{a}^{x_2}{f(x_1)\,dx_1 \cdots dx_{n}}}\right)dx_{n+1}} \\ 
&=\int_{a}^{b}{f(x_{n+1})\left(\sum_{k=1}^{n}{(\rho_{k}(x_{n+1})-\rho_{k}(a))\sum_{\sum{iy_i}=n-k}}{\,\,\prod_{i=1}^{n-k}{(-1)^{y_i}(\rho_{i}(a))^{y_i}}}\right)dx_{n+1}} \\
&=\sum_{k=1}^{n}{\left(\int_{a}^{b}{f(x_{n+1})\rho_{k}(x_{n+1})dx_{n+1}}\right)\sum_{\sum{iy_i}=n-k}}{\,\,\prod_{i=1}^{n-k}{(-1)^{y_i}(\rho_{i}(a))^{y_i}}}\\
&-\sum_{k=1}^{n}{\rho_{k}(a)\left(\int_{a}^{b}{f(x_{n+1})dx_{n+1}}\right)\sum_{\sum{iy_i}=n-k}}{\,\,\prod_{i=1}^{n-k}{(-1)^{y_i}(\rho_{i}(a))^{y_i}}}
.\end{split}
\end{equation*} 
From Lemma \ref{l 2.3}, 
$$
\int_{a}^{b}{f(x_{n+1})\rho_{k}(x_{n+1})\,dx_{n+1}}
=\rho_{k+1}(b)-\rho_{k+1}(a)
,$$
$$
\int_{a}^{b}{f(x_{n+1})\,dx_{n+1}}
=\int_{a}^{b}{f(x_{n+1})\rho_0(x_{n+1})\,dx_{n+1}}
=\rho_{1}(b)-\rho_{1}(a)
.$$
Hence, 
\begin{equation*} 
\begin{split}
J_{n+1,a,b}
&=\sum_{k=1}^{n}{\left(\rho_{k+1}(b)-\rho_{k+1}(a)\right)\sum_{\sum{iy_i}=n-k}}{\,\,\prod_{i=1}^{n-k}{(-1)^{y_i}(\rho_{i}(a))^{y_i}}}\\
&-\sum_{k=1}^{n}{\rho_{k}(a)\left(\rho_{1}(b)-\rho_{1}(a)\right)\sum_{\sum{iy_i}=n-k}}{\,\,\prod_{i=1}^{n-k}{(-1)^{y_i}(\rho_{i}(a))^{y_i}}}\\
&=\sum_{k=2}^{n+1}{\left(\rho_{k}(b)-\rho_{k}(a)\right)\sum_{\sum{iy_i}=n-k+1}}{\,\,\prod_{i=1}^{n-k+1}{(-1)^{y_i}(\rho_{i}(a))^{y_i}}}\\
&+\left(\rho_{1}(b)-\rho_{1}(a)\right)\sum_{k=1}^{n}{\sum_{\sum{iy_i}=n-k}}{(-1)\rho_{k}(a)\,\,\prod_{i=1}^{n-k}{(-1)^{y_i}(\rho_{i}(a))^{y_i}}} 
.\end{split}
\end{equation*} 
Let us define a partition $(Y_1, \ldots, Y_n)$ such that $Y_k=y_k+1$ and, for $i \neq k$, if $1\leq i \leq n-k$, $Y_i=y_i$, otherwise, $Y_i=0$. Noticing that $\sum{i\,Y_i}=\sum{i\,y_i}+k=n$, hence, $(Y_1, \ldots, Y_n)$ is a partition of $n$ for all $k$. Thus, we have 
\begin{equation*}
\begin{split}
&\left(\rho_{1}(b)-\rho_{1}(a)\right)\sum_{k=1}^{n}{\sum_{\sum{iy_i}=n-k}}{(-1)\rho_{k}(a)\,\,\prod_{i=1}^{n-k}{(-1)^{y_i}(\rho_{i}(a))^{y_i}}} \\
&=\left(\rho_{1}(b)-\rho_{1}(a)\right)\sum_{\sum{iY_i}=n}{\,\,\prod_{i=1}^{n}{(-1)^{Y_i}(\rho_{i}(a))^{Y_i}}} \\
&=\sum_{k=1}^{1}{\left(\rho_{k}(b)-\rho_{k}(a)\right)\sum_{\sum{iY_i}=n-k+1}}{\,\,\prod_{i=1}^{n-k+1}{(-1)^{Y_i}(\rho_{i}(a))^{Y_i}}}
.\end{split}
\end{equation*} 
Hence, substituting back, we get the $(n+1)$ case and the theorem is proven. 
\end{proof}
\subsubsection{Definite recurrent integrals}
In this section, we develop two reduction formulae for definite recurrent integrals: the first involves a sum over partitions while the second does not. Then, in the later section, by equating the expressions from these two formulae, we derive partition identities. 
\begin{theorem} \label{t 2.7}
Let $f(x)$ be a function defined in the interval $[a,b]$, for any $n\in\mathbb{N^*}$, we have
\begin{equation*} 
\begin{split}
&\int_{a}^{b}{f(x_n)\int_{a}^{x_n}{f(x_{n-1}) \cdots \int_{a}^{x_2}{f(x_1)\,dx_1 \cdots dx_n}}} \\
&\,\,\,\,=\sum_{k=1}^{n}{\frac{(F(b))^{k}-(F(a))^{k}}{k!}(F(a))^{n-k}\sum_{\sum{iy_i}=n-k}}{\,\,\prod_{i=1}^{n-k}{\frac{(-1)^{y_i}}{i!^{y_i}}}}
.\end{split} 
\end{equation*} 
\end{theorem} 
\begin{remark}
{\em This theorem can also be expressed using the second partition notation:  
}
\begin{equation*} 
\begin{split}
&\int_{a}^{b}{f(x_n)\int_{a}^{x_n}{f(x_{n-1}) \cdots \int_{a}^{x_2}{f(x_1)\,dx_1 \cdots dx_n}}} \\
&\,\,\,\,=\sum_{k=1}^{n}{\sum_{\substack{\sum{\lambda_i}=n-k  \\  \lambda_i \leq \lambda_{i+1}}}{\frac{(F(b))^{k}-(F(a))^{k}}{k!}(F(a))^{n-k}(-1)^{\ell}\,\,\prod_{i=1}^{\ell}{\frac{1}{\lambda_i!}}}}
.\end{split} 
\end{equation*} 
\end{remark}
\begin{proof}
To be concise, we will represent the recurrent integral by $J_{n,a,b}$. From Theorem \ref{t 2.5},
\begin{equation*} 
\begin{split}
J_{n,a,b}=\sum_{k=1}^{n}{(\rho_{k}(b)-\rho_{k}(a))\sum_{\sum{iy_i}=n-k}}{\,\,\prod_{i=1}^{n-k}{(-1)^{y_i}(\rho_{i}(a))^{y_i}}}
.\end{split}
\end{equation*} 
Applying Theorem \ref{t 2.6}, we get 
\begin{equation*} 
\begin{split}
J_{n,a,b}
&=\sum_{k=1}^{n}{\frac{(\rho_{1}(b))^{k}-(\rho_{1}(a))^{k}}{k!}\sum_{\sum{iy_i}=n-k}}{\,\,\prod_{i=1}^{n-k}{(-1)^{y_i}\frac{(\rho_{1}(a))^{iy_i}}{i!^{y_i}}}} \\
&=\sum_{k=1}^{n}{\frac{(\rho_{1}(b))^{k}-(\rho_{1}(a))^{k}}{k!}(\rho_{1}(a))^{n-k}\sum_{\sum{iy_i}=n-k}}{\,\,\prod_{i=1}^{n-k}{\frac{(-1)^{y_i}}{i!^{y_i}}}}
.\end{split} 
\end{equation*}
Noticing that $\rho_{1}(x)$ is simply the primitive $F(x)$ of $f(x)$, we get the theorem. 
\end{proof}
\begin{theorem} \label{t 2.8}
Let $f(x)$ be a function defined in the interval $[a,b]$, for any $n\in\mathbb{N^*}$, we have
\begin{equation*} 
\int_{a}^{b}{f(x_n)\int_{a}^{x_n}{f(x_{n-1}) \cdots \int_{a}^{x_2}{f(x_1)\,dx_1 \cdots dx_n}}}
=\frac{1}{n!}\left(\int_{a}^{b}{f(x)\,dx}\right)^n
=\frac{\left(F(b)-F(a)\right)^n}{n!}
.\end{equation*} 
\end{theorem} 
\begin{proof}
Let $X_i=\int{f(x_i)\,dx_i}=F(x_i)$ and $dX_i=f(x_i)\,dx_i$, we can rewrite the definite recurrent integral as follows (Notice that the limits of integration should also be modified accordingly): 
\begin{equation*}
\int_{a}^{b}{f(x_n)\int_{a}^{x_n}{f(x_{n-1}) \cdots \int_{a}^{x_2}{f(x_1)\,dx_1 \cdots dx_n}}}
=\int_{F(a)}^{F(b)}{\int_{F(a)}^{X_n}{ \cdots \int_{F(a)}^{X_2}{1 \,dX_1 \cdots dX_n}}}.
\end{equation*}
Using Proposition \ref{definite repeated integral of 1}, we get the theorem.  
\end{proof}
\subsubsection{Partition identities derived from recurrent integrals}
By equation the expressions of the definite recurrent integral from Theorem \ref{t 2.7} and Theorem \ref{t 2.8}, we get the following theorem. 
\begin{theorem} \label{t 2.10}
Let $F(x)$ be a function defined in the interval $[a,b]$, for any $n\in\mathbb{N^*}$, we have 
\begin{equation*} 
\sum_{k=1}^{n}{\frac{(F(b))^{k}-(F(a))^{k}}{k!}(F(a))^{n-k}\sum_{\sum{iy_i}=n-k}}{\,\,\prod_{i=1}^{n-k}{\frac{(-1)^{y_i}}{i!^{y_i}}}}
=\frac{\left(F(b)-F(a)\right)^n}{n!}.
\end{equation*} 
Equivalently, we have 
\begin{equation*} 
\sum_{k=1}^{n}{\sum_{\substack{\sum{\lambda_i}=n-k  \\  \lambda_i \leq \lambda_{i+1}}}{\frac{(F(b))^{k}-(F(a))^{k}}{k!}(F(a))^{n-k}(-1)^{\ell}\,\prod_{i=1}^{\ell}{\frac{1}{\lambda_i!}}}}
=\frac{\left(F(b)-F(a)\right)^n}{n!}.
\end{equation*} 
\end{theorem} 
From this general partition identity, we derive some special partition identities. 
\begin{corollary}
For any $n\in\mathbb{N^*}$, we have that 
\begin{equation*} \label{partition1}
\sum_{k=1}^{n}{\frac{1}{k!}\sum_{\sum{iy_i}=n-k}}{\,\,\prod_{i=1}^{n-k}{\frac{(-1)^{y_i}}{i!^{y_i}}}}
=\frac{\left(-1\right)^{n+1}}{n!}.
\end{equation*} 
Equivalently, we have 
\begin{equation*} 
\sum_{\substack{{k+\lambda_1\cdots+\lambda_{\ell}=n} \\ \lambda_i \leq \lambda_{i+1}}}{\frac{(-1)^{\ell+1}}{k!\lambda_1!\cdots\lambda_{\ell}!}}
=\frac{\left(-1\right)^{n}}{n!}.
\end{equation*} 
\end{corollary}
\begin{proof}
Letting $F(x)$ be such that $F(a)=1$ and $F(b)=0$, Theorem \ref{t 2.10} reduces to this corollary. The second equality is obtained by using the second partition notation and combining the two sums into one. 
\end{proof}
\begin{corollary} \label{partition2}
For any $n\in\mathbb{N^*}$ and any $a,b\in\mathbb{C}$, we have that 
\begin{equation*} 
\sum_{k=1}^{n}{\frac{b^k a^{n-k}}{k!}\sum_{\sum{iy_i}=n-k}}{\,\,\prod_{i=1}^{n-k}{\frac{(-1)^{y_i}}{i!^{y_i}}}}
=\frac{(b-a)^n-(-a)^{n}}{n!}.
\end{equation*} 
Equivalently, we have 
\begin{equation*} 
\sum_{\substack{{k+\lambda_1\cdots+\lambda_{\ell}=n} \\ \lambda_i \leq \lambda_{i+1}}}{\frac{b^{k}a^{n-k}}{k!}\frac{(-1)^{\ell}}{\lambda_1!\cdots\lambda_{\ell}!}}
=\sum_{\substack{{k+\lambda_1\cdots+\lambda_{\ell}=n} \\ \lambda_i \leq \lambda_{i+1}}}{\frac{b^{k}}{k!}\frac{(-1)^{\ell}a^{\lambda_1+\cdots+\lambda_{\ell}}}{\lambda_1!\cdots\lambda_{\ell}!}}
=\frac{(b-a)^n-(-a)^{n}}{n!}.
\end{equation*} 
\end{corollary}
\begin{proof}
Letting $F(x)$ be such that $F(a)=a$ and $F(b)=b$, from Theorem \ref{t 2.10}, we have 
\begin{equation*} 
\sum_{k=1}^{n}{\frac{b^{k}a^{n-k}}{k!}\sum_{\sum{iy_i}=n-k}}{\,\,\prod_{i=1}^{n-k}{\frac{(-1)^{y_i}}{i!^{y_i}}}}
-a^n\sum_{k=1}^{n}{\frac{1}{k!}\sum_{\sum{iy_i}=n-k}}{\,\,\prod_{i=1}^{n-k}{\frac{(-1)^{y_i}}{i!^{y_i}}}}
=\frac{\left(b-a\right)^n}{n!}.
\end{equation*} 
Applying Corollary \ref{partition1}, we obtain this corollary. 
The second equality is obtained by using the second partition notation and combining the two sums into one. 
\end{proof}
\begin{corollary} \label{partition3}
For any $n\in\mathbb{N^*}$ and any $b\in\mathbb{C}$, we have that 
\begin{equation*} 
\sum_{k=1}^{n}{\frac{b^k}{k!}\sum_{\sum{iy_i}=n-k}}{\,\,\prod_{i=1}^{n-k}{\frac{(-1)^{y_i}}{i!^{y_i}}}}
=\frac{(b-1)^n-(-1)^{n}}{n!}.
\end{equation*} 
Equivalently, we have 
\begin{equation*} 
\sum_{\substack{{k+\lambda_1\cdots+\lambda_{\ell}=n} \\ \lambda_i \leq \lambda_{i+1}}}{\frac{b^k}{k!}\frac{(-1)^{\ell}}{\lambda_1!\cdots\lambda_{\ell}!}}
=\frac{(b-1)^n-(-1)^{n}}{n!}.
\end{equation*} 
\end{corollary}
\begin{proof}
By setting $a=1$ in Corollary \ref{partition2}, we obtain this corollary. 
\end{proof}
\subsubsection{Generalized recurrent integrals}
The recurrent integral structure is one where all functions are the same. It is analogous to the simple recurrent sum (presented in \cite{RecurrentSums}) where all sequences are the same. However, that last article also presents a more general sum structure where all sequences are distinct. Similarly, we can introduce a generalized recurrent integral structure where all functions are distinct. 
\begin{definition}
We define the general indefinite recurrent integral of order $n$ and denote it as: 
\begin{equation}
J_n[f_1,\ldots, f_n](x)=
\int{f_n(x)\int{f_{n-1}(x)\cdots\int{f_1(x)\,dx^n}}}
.\end{equation}
\end{definition}
\begin{definition}
We define the $n$-th order general definite recurrent integral and denote it as:  
\begin{equation}
J_{n,a,b}[f_1,\ldots,f_n]=
\int_{a}^{b}{f_n(x_n)\int_{a}^{x_n}{f_{n-1}(x_{n-1}) \cdots \int_{a}^{x_2}{f_1(x_1)\,dx_1 \cdots dx_n}}}
.\end{equation}
\end{definition}
In what follows, we conjecture formulas for simplifying each of the indefinite and definite general recurrent integrals. But first, we define the follow notation: 
The permutation group $S_n$ is the set of all permutations of the set $\{1, \ldots, n \}$. Let $\sigma \in S_n$ be a permutation of the set $\{1, \ldots, n \}$ and let $\sigma (i)$ represent the $i$-th element of this given permutation. 
Let $C_{n,k}$ be the set of all combinations (without repetition) of $k$ elements from the set $\{1, \ldots, n \}$. Let $\varepsilon \in C_{n,k}$ be a combination of $k$ elements from the set $\{1, \ldots, n \}$ and let $\varepsilon (i)$ represent the $i$-th element of this given combination.  
\begin{conjecture}
Let $f(x)$ be a function of $x$, for any $n\in\mathbb{N^*}$, we have that 
\begin{equation*}
\begin{split}
&\sum_{\sigma \in S_n}{\left(\int{f_{\sigma(n)}(x)\int{f_{\sigma(n-1)}(x)\cdots\int{f_{\sigma(1)}(x)\,dx^n}}}\right)} \\
&=\prod_{i=1}^{n}{\left(\int{f_i(x)\,dx}\right)}
+\sum_{k=0}^{n-1}{\sum_{\varepsilon \in C_{n,k}}{c_{k,\varepsilon(1),\ldots, \varepsilon(k)}\prod_{i=1}^{k}{\left(\int{f_{\varepsilon(i)}(x)dx}\right)}}}
\end{split}
\end{equation*}
where the $c_{k,\varepsilon(1),\ldots, \varepsilon(k)}$ are constants of integration. 
\end{conjecture}
\begin{example}
For $f_1(x)=e^x$ and $f_2(x)=\sin x$, we have 
\begin{equation*}
\begin{split}
&\int{\sin x \int{e^x \,dx^2}}+\int{e^x \int{\sin x \,dx^2}} \\
&=\left(\int{e^x}\,dx\right)\left(\int{\sin x}\,dx\right)
+c_{1,1}\left(\int{\sin x}\,dx\right)+c_{1,2}\left(\int{e^x}\,dx\right)
+c_0 \\
&=-e^x \cos x - c_{1,1}\cos x + c_{1,2}e^x + c_0
.\end{split}
\end{equation*}
\end{example}
\begin{conjecture}
Let $f(x)$ be a function defined in the interval $[a,b]$, for any $n\in\mathbb{N^*}$, we have 
\begin{equation*}
\sum_{\sigma \in S_n}{\left(\int_{a}^{b}{f_{\sigma(n)}(x_n)\int_{a}^{x_n}{f_{\sigma(n-1)}(x_{n-1}) \cdots \int_{a}^{x_2}{f_{\sigma(1)}(x_1)\,dx_1 \cdots dx_n}}}\right)}
=\prod_{i=1}^{n}{\left(\int_{a}^{b}{f_i(x)\,dx}\right)}
.\end{equation*}
\end{conjecture}
\begin{example} 
For $f_1(x)=e^x$ and $f_2(x)=\sin x$, we have 
\begin{equation*}
\begin{split}
&\int_{a}^{b}{\sin x_2 \int_{a}^{x_2}{e^{x_1} \,dx_1dx_2}}+\int_{a}^{b}{e^x_2 \int_{a}^{x_2}{\sin x_1 \,dx_1dx_2}} \\
&=\left(\int_{a}^{b}{e^x}\,dx\right)\left(\int_{a}^{b}{\sin x}\,dx\right)
=-e^b \cos b -e^a \cos a +e^b \cos a  +e^a \cos b    
.\end{split}
\end{equation*}
\end{example}
\section{Explicit formula for the $n$-th integral of $x^m(\ln x)^{m'}$} \label{section 3}
An integral table can be found at the back of any calculus textbook. Such integral tables contain general solutions to some general integrals. However, one integral that cannot be found is the $n$-th order integral of $x^m(\ln x)^{m'}$,
$$\int{x^m(\ln ⁡x)^{m'}dx^n}.$$ 
This integral is of great importance due to its connection with multiple harmonic star sums (MHSSs). In this section, we develop an explicit formula for this general integral. Through the formulas we derive in this section, we further contribute to strengthening the connection between logarithmic integrals and multiple harmonic star sums. We derive two expressions for this integral: the first in terms of shifted MHSSs and the second in terms of a sum over partitions of shifted harmonic sums. Hence, further highlighting the number theoretical nature of this integral. But, before we begin, for the sake of conciseness, let us define the following notation: 
$$
\int{f(x)\,dx^n}
=\int{f(x)\,d^n x}
=\underbrace{\int\cdots\int}_{n} \,f(x)\,\underbrace{dx\cdots dx}_{n}.
$$
\subsection{Lemmas}
Before we proceed to prove the needed lemmas, we need to present the gamma function. The gamma function \cite{Gamma,eulerGamma}, denoted $\Gamma (z)$, is an extension of the factorial to complex numbers excluding the non-positive integers as it satisfies the following property: $\Gamma(z+1)=z\Gamma(z)$. It was developed by Euler and is defined by the following expression: 
$$\Gamma(z)=\int_{0}^{\infty}{x^{z-1}e^{-x}\,dx}.$$
If $z\in\mathbb{N}$, the gamma function reduces back to the following factorial: 
$
\Gamma(z)=(z-1)!.
$ \\
We will denote by $\mathbb{Z^-}$ the set of negative integers. Hence, $\Gamma(z+1)$ is defined for $z\in \mathbb{C}\setminus \mathbb{Z^-}$. 
\\

Now that we have introduced the gamma function, we proceed in proving the lemmas. The first lemma consists of proving an expression for a simpler version of the integral. 
\begin{lemma} \label{l 3.3}
For any $m\in \mathbb{C}\setminus \{-1\}$ and for any $m'\in\mathbb{N}$, we have that 
\begin{equation*}
\int{x^m (\ln x)^{m'}\,dx}
=m'!\frac{x^{m+1}}{m+1}\sum_{k=0}^{m'}{(-1)^{m'-k}\frac{(\ln x)^k}{k!}\left(\frac{1}{(m+1)^{m'-k}}\right)}+C_1
.\end{equation*}
\end{lemma}
\begin{proof}
1. Base case: verify true for $m'=0$. 
\begin{equation*}
\begin{split}
0!\frac{x^{m+1}}{m+1}\sum_{k=0}^{0}{(-1)^{-k}\frac{(\ln x)^k}{k!}\left(\frac{1}{(m+1)^{-k}}\right)}
&=\frac{x^{m+1}}{m+1}
=\int{x^m\,dx}
.\end{split}
\end{equation*}
2. Induction hypothesis: assume the statement is true until $m'$. 
\begin{equation*}
\int{x^m (\ln x)^{m'}\,dx}
=m'!\frac{x^{m+1}}{m+1}\sum_{k=0}^{m'}{(-1)^{m'-k}\frac{(\ln x)^k}{k!}\left(\frac{1}{(m+1)^{m'-k}}\right)}
.\end{equation*}
3. Induction step: we will show that this statement is true for $(m'+1)$. \\
We have to show the following statement to be true: 
\begin{equation*}
\int{x^m (\ln x)^{m'+1}\,dx}
=(m'+1)!\frac{x^{m+1}}{m+1}\sum_{k=0}^{m'+1}{(-1)^{m'-k+1}\frac{(\ln x)^k}{k!}\left(\frac{1}{(m+1)^{m'-k+1}}\right)}
.\end{equation*}
$$ \\ $$ 
We set 
$
\begin{cases}
u= (\ln x)^{m'+1}, &  du=(m'+1)\left(\frac{1}{x}\right)(\ln x)^{m'} \,dx\\
v= \frac{x^{m+1}}{m+1}, & dv=x^m \,dx.
\end{cases}
$ \\
Applying integration by parts, we get 
\begin{equation*}
\begin{split}
\int{x^m (\ln x)^{m'+1}\,dx}
&=\frac{x^{m+1}}{m+1}(\ln x)^{m'+1}-\frac{m'+1}{m+1}\int{x^{m}(\ln x)^{m'}\,dx}
.\end{split}
\end{equation*}
Applying the induction hypothesis, 
\begin{equation*}
\begin{split}
&\int{x^m (\ln x)^{m'+1}\,dx} \\
&\,\,\,\,=\frac{x^{m+1}}{m+1}(\ln x)^{m'+1}-\frac{m'+1}{m+1}m'!\frac{x^{m+1}}{m+1}\sum_{k=0}^{m'}{(-1)^{m'-k}\frac{(\ln x)^k}{k!}\left(\frac{1}{(m+1)^{m'-k}}\right)}\\
&\,\,\,\,=\frac{x^{m+1}}{m+1}(\ln x)^{m'+1}+(m'+1)!\frac{x^{m+1}}{m+1}\sum_{k=0}^{m'}{(-1)^{m'-k+1}\frac{(\ln x)^k}{k!}\left(\frac{1}{(m+1)^{m'-k+1}}\right)} \\
&\,\,\,\,=(m'+1)!\frac{x^{m+1}}{m+1}\sum_{k=0}^{m'+1}{(-1)^{m'-k+1}\frac{(\ln x)^k}{k!}\left(\frac{1}{(m+1)^{m'-k+1}}\right)}
.\end{split}
\end{equation*}
Hence, the lemma is proven by induction. 
\end{proof}
In \cite{RecurrentSums}, the author proved the following theorem (called the variation formula for recurrent sums), 
\begin{equation} \label{var}
\sum_{N_m=q}^{n+1}{\cdots \sum_{N_2=q}^{N_3}{\sum_{N_1=q}^{N_2}{a_{N_m}\cdots a_{N_2}a_{N_1}}}}
=\sum_{k=0}^{m}{\left(a_{n+1}\right)^{m-k}\left(\sum_{N_k=q}^{n}{\cdots \sum_{N_2=q}^{N_3}{\sum_{N_1=q}^{N_2}{a_{N_k}\cdots a_{N_2}a_{N_1}}}}\right)}.
\end{equation}
This relation needs to be adapted to a certain form in order for it to be used in the proof of Theorem \ref{t 3.1}. This needed adaptation is illustrated by the following lemma. This lemma consists of a multiple harmonic star sum identity expressing a MHSS in terms of lower order MHSSs. 
\begin{lemma} \label{l 3.4}
We have that 
\begin{equation*}
\begin{split}
&\sum_{N_{m'-i}=1+m}^{n+m+1}{\cdots \sum_{N_{2}=1+m}^{N_3}{\sum_{N_{1}=1+m}^{N_2}{\frac{1}{N_{m'-i}} \cdots \frac{1}{N_2}\frac{1}{N_1}}}} \\
&\,\,\,\,=\sum_{k=i}^{m'}{\frac{1}{(m+n+1)^{k-i}}\sum_{N_{m'-k}=1+m}^{n+m}{\cdots \sum_{N_{2}=1+m}^{N_3}{\sum_{N_{1}=1+m}^{N_2}{\frac{1}{N_{m'-k}} \cdots \frac{1}{N_2}\frac{1}{N_1}}}}}
.\end{split}
\end{equation*} 
\end{lemma}
\begin{remark}
{\em Using the notation, we can write the lemma as follows: }
\begin{equation*}
\zeta^{\star}_{1+m,n+m+1}(\{1\}_{m'-i})
=\sum_{k=i}^{m'}{\frac{\zeta^{\star}_{1+m,n+m}(\{1\}_{m'-k})}{(n+m+1)^{k-i}}}
.\end{equation*}
\end{remark}
\begin{proof}
From equation (\ref{var}) with $m$ substituted by $m'-i$, with $a_{N_k}=\frac{1}{N_k}$, with $q=1+m$, and with $n$ substituted by $n+m$, we have that 
\begin{equation*}
\begin{split}
&\sum_{N_{m'-i}=1+m}^{n+m+1}{\cdots \sum_{N_2=1+m}^{N_3}{\sum_{N_1=1+m}^{N_2}{\frac{1}{N_{m'-i}}\cdots \frac{1}{N_2}\frac{1}{N_1}}}} \\
&\,\,\,\,=\sum_{k=0}^{m'-i}{\frac{1}{\left(m+n+1\right)^{m'-k-i}}\left(\sum_{N_k=1+m}^{n+m}{\cdots \sum_{N_2=1+m}^{N_3}{\sum_{N_1=1+m}^{N_2}{\frac{1}{N_k}\cdots \frac{1}{N_2}\frac{1}{N_1}}}}\right)}.
\end{split}
\end{equation*}
For the sake of conciseness, let us denote the first term as $S$. Substituting $k$ by $m'-k$, we get 
\begin{equation*}
\begin{split}
S=\sum_{m'-k=m'-i}^{0}{\frac{1}{\left(m+n+1\right)^{k-i}}\left(\sum_{N_{m'-k}=1+m}^{n+m}{\cdots \sum_{N_2=1+m}^{N_3}{\sum_{N_1=1+m}^{N_2}{\frac{1}{N_{m'-k}}\cdots \frac{1}{N_2}\frac{1}{N_1}}}}\right)}.
\end{split}
\end{equation*}
Noticing that the interval $m'-i \leq m'-k \leq 0$ is equivalent to the interval $i \leq k \leq m'$, 
\begin{equation*}
\begin{split}
S=\sum_{k=i}^{m'}{\frac{1}{\left(m+n+1\right)^{k-i}}\left(\sum_{N_{m'-k}=1+m}^{n+m}{\cdots \sum_{N_2=1+m}^{N_3}{\sum_{N_1=1+m}^{N_2}{\frac{1}{N_{m'-k}}\cdots \frac{1}{N_2}\frac{1}{N_1}}}}\right)}.
\end{split}
\end{equation*} 
The relation is proven. 
\end{proof}
\subsection{Theorems}
Now that the needed lemmas have been proven, we can develop the following explicit expression for the $n$-th integral of $x^m (\ln x)^{m'}$ in terms of a shifted multiple harmonic star sum. 
\begin{theorem} \label{t 3.1}
For any $m\in \mathbb{C}\setminus\mathbb{Z^-}$, any $m' \in \mathbb{N}$, and any $n \in \mathbb{N^*}$, we have that 
\begin{dmath*}
\int {x^m (\ln x)^{m'} d^n x}
=(m')!\Gamma(m+1)\frac{x^{n+m}}{\Gamma(n+m+1)}
\left[\sum_{k=0}^{m'}{(-1)^{m'-k} \frac{(\ln x)^k}{k!} \zeta_{1+m,n+m}^{\star}(\{1\}_{m'-k})}\right]
+C_n
\end{dmath*}
where 
\begin{equation*}
\zeta^{\star}_{1+m,n+m}(\{1\}_{m'-k})
=\sum_{N_{m'-k}=1+m}^{n+m} {\cdots \sum_{N_2=1+m}^{N_3}{\sum_{N_1=1+m}^{N_2}{\frac{1}{N_{m'-k}} \cdots \frac{1}{N_2}\frac{1}{N_1}}}}. 
\end{equation*}
\end{theorem}
\begin{remark}
{\em If $m  \in \mathbb{N}$, we replace the gamma function by the corresponding factorial. }
\begin{dmath*}
\int {x^m (\ln x)^{m'} d^n x}
=(m)!(m')!\frac{x^{n+m}}{(n+m)!}
\left[\sum_{k=0}^{m'}{(-1)^{m'-k} \frac{(\ln x)^k}{k!} \zeta^{\star}_{1+m,n+m}(\{1\}_{m'-k})}\right]
+C_n
.\end{dmath*}
\end{remark}
\begin{remark}
{\em If $m  \in \mathbb{Z^-}$ but $(n+m)<0$, by taking the limit of the gamma functions ratio, }
\begin{dmath*}
\int {x^m (\ln x)^{m'} d^n x}
=(m')!\frac{x^{n+m}}{(m+1)\cdots (n+m)}
\left[\sum_{k=0}^{m'}{(-1)^{m'-k} \frac{(\ln x)^k}{k!} \zeta^{\star}_{1+m,n+m}(\{1\}_{m'-k})}\right]
+C_n
.\end{dmath*}
\end{remark}
\begin{remark}
{\em As the reason for the constants of integration represented by $C_n$ is trivial, we will neglect this term in the proof of the theorem for simplicity. }
\end{remark}
\begin{proof}
1. Base case: verify true for $n=1$. 
\begin{equation*}
\begin{split}
&\Gamma(m+1)(m')!\frac{x^{1+m}}{\Gamma(2+m)}
\left[\sum_{k=0}^{m'}{(-1)^{m'-k} \frac{(\ln x)^k}{k!} \zeta_{1+m,1+m}^{\star}(\{1\}_{m'-k})}\right] \\
&=m'!\frac{x^{m+1}}{m+1}\sum_{k=0}^{m'}{(-1)^{m'-k}\frac{(\ln x)^k}{k!}\left(\frac{1}{(m+1)^{m'-k}}\right)}
.\end{split}
\end{equation*}
Likewise, from Lemma \ref{l 3.3}, 
\begin{equation*}
\int{x^m (\ln x)^{m'}\,dx}
=m'!\frac{x^{m+1}}{m+1}\sum_{k=0}^{m'}{(-1)^{m'-k}\frac{(\ln x)^k}{k!}\left(\frac{1}{(m+1)^{m'-k}}\right)}
.\end{equation*}
2. Induction hypothesis: assume the statement is true until $n$. 
\begin{dmath*}
\int {x^m (\ln x)^{m'} d^n x}
=\Gamma(m+1)(m')!\frac{x^{n+m}}{\Gamma(n+m+1)}
\left[\sum_{k=0}^{m'}{(-1)^{m'-k} \frac{(\ln x)^k}{k!} \zeta_{1+m,n+m}^{\star}(\{1\}_{m'-k})}\right]
.\end{dmath*}
3. Induction step: we will show that this statement is true for $(n+1)$. \\
We have to show the following statement to be true 
\begin{dmath*}
\int {x^m (\ln x)^{m'} d^{n+1} x}
=\Gamma(m+1)(m')!\frac{x^{n+m+1}}{\Gamma(n+m+2)}
\left[\sum_{k=0}^{m'}{(-1)^{m'-k} \frac{(\ln x)^k}{k!} \zeta_{1+m,n+m+1}^{\star}(\{1\}_{m'-k})}\right]
.\end{dmath*}
$$ \\ $$ 
Using the induction hypothesis, 
\begin{equation*}
\begin{split}
\int {x^m (\ln x)^{m'} d^{n+1} x}
&=\int {\left(\int{x^m (\ln x)^{m'} d^{n}x}\right) dx}\\
&=\int {\Gamma(m+1)\frac{(m')! x^{n+m}}{\Gamma(n+m+1)}
\left[\sum_{k=0}^{m'}{(-1)^{m'-k} \frac{(\ln x)^k}{k!} \zeta_{1+m,n+m}^{\star}(\{1\}_{m'-k})}\right] dx} \\
&=\frac{\Gamma(m+1)(m')!}{\Gamma(n+m+1)}
\sum_{k=0}^{m'}{ \frac{(-1)^{m'-k}}{k!} \zeta_{1+m,n+m}^{\star}(\{1\}_{m'-k})} \left(\int {x^{n+m}(\ln x)^k \,dx}\right) 
.\end{split}
\end{equation*}
Using Lemma \ref{l 3.3}, we have  
\begin{equation*}
\int{x^{n+m} (\ln x)^{k}\,dx}
=k!\frac{x^{m+n+1}}{m+n+1}\sum_{i=0}^{k}{(-1)^{k-i}\frac{(\ln x)^i}{i!}\left(\frac{1}{(m+n+1)^{k-i}}\right)}
.\end{equation*}
Substituting back and simplifying, we get 
\begin{equation*}
\begin{split}
\int{x^m(\ln x)^{m'}\,dx}
=\frac{m'!\Gamma(m+1)x^{m+n+1}}{\Gamma(m+n+2)}\sum_{k=0}^{m'}{\sum_{i=0}^{k}{(-1)^{m'-i}\frac{(\ln x)^i}{i!}\left\{\frac{\zeta_{1+m,n+m}^{\star}(\{1\}_{m'-k})}{(m+n+1)^{k-i}}\right\}}}
.\end{split}
\end{equation*}
Now we need to invert the order of summation. From Theorem 3.1 of \cite{RecurrentSums}, we have 
$$
\sum_{k=0}^{m'}{b_k\sum_{i=0}^{k}{a_i}}
=\sum_{i=0}^{m'}{a_i\sum_{k=i}^{m'}{b_k}}
.$$
Let 
$$
a_i=(-1)^{m'-i}\frac{(\ln x)^i}{i!}\frac{1}{(m+n+1)^{-i}}, 
\,\,\,\,
b_k=\frac{1}{(m+n+1)^k}\zeta_{1+m,n+m}^{\star}(\{1\}_{m'-k})
.$$
Hence, applying this formula to interchange the order of summation, we get 
\begin{equation*}
\begin{split}
&\sum_{k=0}^{m'}{\sum_{i=0}^{k}{(-1)^{m'-i}\frac{(\ln x)^i}{i!}\left\{\frac{\zeta_{1+m,n+m}^{\star}(\{1\}_{m'-k})}{(m+n+1)^{k-i}}\right\}}}\\
&=\sum_{i=0}^{m'}{(-1)^{m'-i}\frac{(\ln x)^i}{i!}\frac{1}{(m+n+1)^{-i}}\sum_{k=i}^{m'}{\frac{\zeta_{1+m,n+m}^{\star}(\{1\}_{m'-k})}{(m+n+1)^k}}} \\
&=\sum_{i=0}^{m'}{(-1)^{m'-i}\frac{(\ln x)^i}{i!}\sum_{k=i}^{m'}{\frac{\zeta_{1+m,n+m}^{\star}(\{1\}_{m'-k})}{(m+n+1)^{k-i}}}} 
.\end{split}
\end{equation*}
Applying Lemma \ref{l 3.4}, we get 
\begin{equation*}
\begin{split}
\sum_{k=0}^{m'}{\sum_{i=0}^{k}{(-1)^{m'-i}\frac{(\ln x)^i}{i!}\left\{\frac{\zeta_{1+m,n+m}^{\star}(\{1\}_{m'-k})}{(m+n+1)^{k-i}}\right\}}}
=\sum_{i=0}^{m'}{(-1)^{m'-i}\frac{(\ln x)^i}{i!}\left\{\zeta_{1+m,n+m+1}^{\star}(\{1\}_{m'-i})\right\}} 
.\end{split}
\end{equation*}
Substituting back into the expression of the integral, we get 
\begin{dmath*}
\int {x^m (\ln x)^{m'} d^{n+1} x}
=\Gamma(m+1)(m')!\frac{x^{n+m+1}}{\Gamma(n+m+2)}
\left[\sum_{i=0}^{m'}{(-1)^{m'-i} \frac{(\ln x)^i}{i!} \zeta_{1+m,n+m+1}^{\star}(\{1\}_{m'-i})}\right]
.\end{dmath*}
Hence, the case for $(n+1)$ is proven. Thus, the relation is proven by induction. 
\end{proof}
Before we proceed to prove a second expression for this logarithmic integral, we present a few corollaries that will be useful later. 
\begin{corollary} \label{l 3.1}
For any $m\in \mathbb{C} \setminus \{-1\}$, we have that 
\begin{equation*}
\int {x^m \ln x \,dx}
=\frac{x^{m+1}}{m+1} \left[ \ln x -\frac{1}{m+1} \right] +C_1.
\end{equation*}
\end{corollary}

\begin{corollary} \label{l 3.2}
For any $m\in \mathbb{C}\setminus\mathbb{Z^-}$ and for any $n\in\mathbb{N^*}$, we have that 
\begin{equation*}
\int {x^m \ln x \,d^n x}
=\Gamma(m+1)\frac{x^{n+m}}{\Gamma(n+m+1)} \left[ \ln x - \sum_{N=1+m}^{n+m}{\frac{1}{N}} \right]
+C_{n}.
\end{equation*} 
\end{corollary}
\begin{remark}
{\em If $m  \in \mathbb{N}$, we replace the gamma function by the corresponding factorial. }
\begin{equation*}
\int {x^m \ln x \,d^n x}
=m!\frac{x^{n+m}}{(n+m)!} \left[ \ln x - \sum_{N=1+m}^{n+m}{\frac{1}{N}} \right]
+C_{n}.
\end{equation*} 
\end{remark}
\begin{remark}
{\em If $m  \in \mathbb{Z^-}$ but $(n+m)<0$, by taking the limit of the gamma functions ratio, }
\begin{equation*}
\int {x^m \ln x \,d^n x}
=\frac{x^{n+m}}{(n+m)\cdots(m+1)} \left[ \ln x - \sum_{N=1+m}^{n+m}{\frac{1}{N}} \right]
+C_{n}.
\end{equation*} 
\end{remark}

Using the reduction theorem for recurrent sums presented in \cite{RecurrentSums}, we can develop the following explicit expression for the $n$-th integral of $x^m (\ln x)^{m'}$ in terms of a sum over partitions. 
\begin{theorem} \label{t 3.2}
For any $m\in \mathbb{C}\setminus\mathbb{Z^-}$, any $m' \in \mathbb{N}$, and any $n \in \mathbb{N^*}$, we have that 
\begin{dmath*}
\int {x^m (\ln x)^{m'} d^n x}
=\frac{(m')!\Gamma(m+1) x^{n+m}}{\Gamma(n+m+1)}
\left[\sum_{k=0}^{m'}{(-1)^{m'-k} \frac{(\ln x)^k}{k!}\sum_{\substack{k \\ \sum{i \, y_{k,i}}={m'-k}}}{\prod_{i=1}^{m'-k}{\frac{1}{y_{k,i}! i^{y_{k,i}}}\left( \sum_{N=1+m}^{n+m}{\frac{1}{N^i}} \right)^{y_{k,i}}}}}\right]
+C_n
.\end{dmath*}
\end{theorem}
\begin{remark}
{\em If $m  \in \mathbb{N}$, we replace the gamma function by the corresponding factorial. }
\begin{dmath*}
\int {x^m (\ln x)^{m'} d^n x}
=(m)!(m')!\frac{x^{n+m}}{(n+m)!}
\left[\sum_{k=0}^{m'}{(-1)^{m'-k} \frac{(\ln x)^k}{k!}\sum_{\substack{k \\ \sum{i \, y_{k,i}}={m'-k}}}{\prod_{i=1}^{m'-k}{\frac{1}{y_{k,i}! i^{y_{k,i}}}\left( \sum_{N=1+m}^{n+m}{\frac{1}{N^i}} \right)^{y_{k,i}}}}}\right]
+C_n
.\end{dmath*}
\end{remark}
\begin{remark}
{\em If $m  \in \mathbb{Z^-}$ but $(n+m)<0$, by taking the limit of the gamma functions ratio, }
\begin{dmath*}
\int {x^m (\ln x)^{m'} d^n x}
=\frac{(m')!x^{n+m}}{(n+m)\cdots(m+1)}
\left[\sum_{k=0}^{m'}{(-1)^{m'-k} \frac{(\ln x)^k}{k!}\sum_{\substack{k \\ \sum{i \, y_{k,i}}={m'-k}}}{\prod_{i=1}^{m'-k}{\frac{1}{y_{k,i}! i^{y_{k,i}}}\left( \sum_{N=1+m}^{n+m}{\frac{1}{N^i}} \right)^{y_{k,i}}}}}\right]
+C_n
.\end{dmath*}
\end{remark}
\begin{proof}
In \cite{RecurrentSums}, the author proved the following reduction formula, 
$$\sum_{N_m=q}^{n}{\cdots \sum_{N_1=q}^{N_2}{a_{N_m}\cdots a_{N_1}}}
=\sum_{\substack{k \\ \sum{i.y_{k,i}}=m}}{\prod_{i=1}^{m}{\frac{1}{(y_{k,i})!} \left( \frac{1}{i}\sum_{N=q}^{n}{(a_N)^i }\right)^{y_{k,i}}}}
.$$
Applying this formula (with $m$ substituted by $m'$, $n$ substituted by $n+m$, $q=1+m$, and $a_{N_k}=\frac{1}{N_k}$) to Theorem \ref{t 3.1}, we obtain this theorem. 
\end{proof}
Logarithms and logarithmic integrals have always held a special place in number theory. The greatest symbol of this connection is the prime number theorem. This connection is further highlighted here as a logarithmic integral was shown to be related to two major number theoretical concepts: MHSSs in Theorem \ref{t 3.1} and partitions in Theorem \ref{t 3.2}. \\ 

Although it is not of major importance, we can use the theorems previously developed for definite repeated integrals to develop some formulae for the $n$-th definite integral of $x^m(\ln x)^{m'}$. 
\begin{definition}
We define $L_{n,m,m'}(x)$ as the $n$-th primitive of $x^m (\ln x)^{m'}$ whose expression can be seen from Theorem \ref{t 3.1} or Theorem \ref{t 3.2}. 
\end{definition}
\begin{theorem} \label{t 3.3}
For any $m\in \mathbb{C}\setminus\mathbb{Z^-}$, any $m' \in \mathbb{N}$, and any $n \in \mathbb{N^*}$, we have that 
\begin{equation*}
\int_{a}^{b}{\cdots \int_{a}^{x_3}{\int_{a}^{x_2}{x_1^m (\ln (x_1))^{m'}dx_1dx_2 \cdots dx_n}}}
=L_{n,m,m'}(b)-\sum_{k=1}^{n}{(b-a)^{n-k}L_{k,m,m'}(a)}
.\end{equation*}
\end{theorem}
\begin{proof}
By applying Theorem \ref{t 2.1} for $f(x)=x^m (\ln x)^{m'}$, we get this theorem. 
\end{proof}
\begin{theorem} \label{t 3.4}
For any $m\in \mathbb{C}\setminus\mathbb{Z^-}$, any $m' \in \mathbb{N}$, and any $n \in \mathbb{N^*}$, we have that 
\begin{equation*}
\int_{a}^{b}{\cdots \int_{a}^{x_3}{\int_{a}^{x_2}{x_1^m (\ln (x_1))^{m'}dx_1dx_2 \cdots dx_n}}}
=\frac{b^{n-1}}{(n-1)!}\sum_{k=0}^{n-1}{\binom{n-1}{k}\frac{(-1)^k}{b^k}\int_{a}^{b}{x^{m+k} (\ln x)^{m'}\,dx}}
.\end{equation*}
\end{theorem}
\begin{proof}
By applying Theorem \ref{definite repeated integral} for $f(x)=x^m (\ln x)^{m'}$, we get this theorem. 
\end{proof}
\begin{theorem} \label{t 3.5}
For any $m\in \mathbb{C}\setminus\mathbb{Z^-}$, any $m' \in \mathbb{N}$, and any $n \in \mathbb{N^*}$, we have that 
\begin{equation*}
\begin{split}
&\int_{a}^{b}{\cdots \int_{a}^{x_3}{\int_{a}^{x_2}{x_1^m (\ln (x_1))^{m'}dx_1dx_2 \cdots dx_n}}}\\
&\,\,=\frac{b^{n-1}}{(n-1)!}\sum_{k=0}^{n-1}{\binom{n-1}{k}\frac{(-1)^k}{b^k}\sum_{j=0}^{m'}{\frac{(-1)^{m'-j}}{(m+k+1)^{m'-j+1}}\frac{m'!}{j!}\left[b^{m+k+1}(\ln b)^j - a^{m+k+1}(\ln a)^j\right]}}
.\end{split}
\end{equation*}
\end{theorem}
\begin{proof}
By applying Lemma \ref{l 3.3} to Theorem \ref{t 3.4}, we get this theorem. 
\end{proof}
\subsection{Particular cases}
In what follows, we illustrate some special particular cases of Theorem \ref{t 3.1}. 
\begin{itemize}
\item For $n=1$: 
\begin{equation}
\int{x^m (\ln x)^{m'}\,dx}
=m'!\frac{x^{m+1}}{m+1}\sum_{k=0}^{m'}{(-1)^{m'-k}\frac{(\ln x)^k}{k!}\left(\frac{1}{(m+1)^{m'-k}}\right)}+C_1
.\end{equation}
\item For $n=1$ and $m=0$:
\begin{equation}
\int{(\ln x)^{m'}\,dx}
=m'!x\sum_{k=0}^{m'}{(-1)^{m'-k}\frac{(\ln x)^k}{k!}}+C_1
.\end{equation}
\item For $m=0$:
\begin{equation}
\int {(\ln x)^{m'} d^n x}
=m'!\frac{x^{n}}{n!}
\sum_{k=0}^{m'}{(-1)^{m'-k} \frac{(\ln x)^k}{k!} \zeta_{n}^{\star}(\{1\}_{m'-k})}
+C_n,
\end{equation}
where 
$$
\zeta_{n}^{\star}(\{1\}_{m'-k})
=\sum_{N_{m'-k}=1}^{n} {\cdots \sum_{N_2=1}^{N_3}{\sum_{N_1=1}^{N_2}{\frac{1}{N_{m'-k}} \cdots \frac{1}{N_2}\frac{1}{N_1}}}}.
$$
\item For $m'=1$: 
\begin{equation}
\int{x^m (\ln x)\,d^n x}
=\Gamma(m+1)\frac{x^{n+m}}{\Gamma(n+m+1)}\left[\ln x -\sum_{N=1+m}^{n+m}{\frac{1}{N}}\right]+C_n
.\end{equation}
\item For $n=1$ and $m'=1$:
\begin{equation}
\int{x^m (\ln x)\,dx}
=\frac{x^{m+1}}{m+1}\left[\ln x -\frac{1}{m+1}\right]+C_1
.\end{equation}
\item For $m'=0$:
\begin{equation}
\int{x^m \,d^n x}
=\frac{x^{n+m}}{\Gamma(n+m+1)}+C_n
.\end{equation}
\end{itemize}
\section{Application to the expression of harmonic sums} \label{section 4}
In this section, we use the connection between logarithmic integrals and harmonic sums established in the previous section to prove some results related to the harmonic sum, repeated harmonic sum, and binomial-harmonic sum. 
\subsection{An alternating expression for the harmonic sum}
In this section, we derive a new expression for the harmonic sum as an alternating sum. We present a particularly original proof for this expression using the logarithmic integral from Corollary \ref{l 3.2}. 

We begin by proving the following lemma. 
\begin{lemma} \label{l 4.1}
For any $n\in\mathbb{N^*}$ and any $m\in\mathbb{N}$, we have that 
\begin{equation*}
\frac{d^n}{dx^n}(x^{n+m}\ln x)
=\frac{(n+m)!}{m!}x^m \ln x +n!x^m\sum_{k=1}^{n}{\frac{(-1)^{k+1}}{k}\binom{m+n}{m+k}}
.\end{equation*}
\end{lemma}
\begin{proof}
Let $u=x^{n+m}$, $v=\ln x$. 
From Leibniz's theorem, we get 
\begin{equation*}
\frac{d^n}{dx^n}\left(x^{n+m}\ln x\right)
=\left(u v\right)^{(n)}
=\sum_{k=0}^{n}{\binom{n}{k}u^{(n-k)}v^{(k)}}
=\sum_{k=0}^{n}{\binom{n}{k}[x^{n+m}]^{(n-k)}[\ln x]^{(k)}}
.\end{equation*}
For $k=0$, $\binom{n}{k}[x^{n+m}]^{(n-k)}[\ln x]^{(k)}=\frac{(n+m)!}{m!}x^m \ln x$. Hence, 
\begin{equation*}
\frac{d^n}{dx^n}\left(x^{n+m}\ln x\right)
=\frac{(n+m)!}{m!}x^m \ln x+\sum_{k=1}^{n}{\binom{n}{k}[x^{n+m}]^{(n-k)}[\ln x]^{(k)}}
.\end{equation*}
We also know the following: 
\begin{equation*}
\forall z \in \mathbb{N}, 
[x^{n+m}]^{(z)}=\frac{(m+n)!}{(m+n-z)!}x^{m+n-z},
\text{ hence, }
[x^{n+m}]^{(n-k)}=\frac{(m+n)!}{(m+k)!}x^{m+k}. 
\end{equation*}
\begin{equation*}
\forall k \in \mathbb{N^*}, v^{(k)}=\frac{(-1)^{k+1}(k-1)!}{x^k}. 
\end{equation*}
Hence, by substituting back, we get 
\begin{equation*}
\begin{split}
\frac{d^n}{dx^n}\left(x^{n+m}\ln x\right)
&=\frac{(n+m)!}{m!}x^m \ln x+\sum_{k=1}^{n}{\binom{n}{k}\left[\frac{(m+n)!}{(m+k)!}x^{m+k}\right]\left[\frac{(-1)^{k+1}(k-1)!}{x^k}\right]} \\
&=\frac{(n+m)!}{m!}x^m \ln x +x^m\sum_{k=1}^{n}{\frac{(-1)^{k+1}}{k}\frac{(m+n)!n!}{(m+k)!(n-k)!}} \\
&=\frac{(n+m)!}{m!}x^m \ln x +n!x^m\sum_{k=1}^{n}{\frac{(-1)^{k+1}}{k}\binom{m+n}{m+k}}
.\end{split}
\end{equation*}
The formula is proven. 
\end{proof}
Using the lemma, we develop an alternating sum expression for the harmonic sum.  
\begin{theorem} \label{t 4.1}
For any $n\in\mathbb{N^*}$ and any $m\in\mathbb{N}$, we have that 
\begin{equation*}
\binom{m+n}{m}\sum_{k=1+m}^{n+m}{\frac{1}{k}}
=\sum_{k=1}^{n}{\frac{(-1)^{k+1}}{k}\binom{m+n}{m+k}}
.\end{equation*}
\end{theorem}
\begin{proof}
From Corollary \ref{l 3.2}, we have that 
\begin{equation*}
\int {x^m \ln x \,dx^n}
=m!\frac{x^{n+m}}{(n+m)!} \left[ \ln x - \sum_{N=1+m}^{n+m}{\frac{1}{N}} \right]
+C_{n}
.\end{equation*} 
We rearrange the terms as follows, 
\begin{equation*}
x^{n+m}\sum_{N=1+m}^{n+m}{\frac{1}{N}}
=x^{n+m}\ln x - \frac{(n+m)!}{m!}\int{x^m \ln x \,dx^n}+C_n
.\end{equation*}
By differentiating $n$ times, we get 
\begin{equation*}
\left(\sum_{N=1+m}^{n+m}{\frac{1}{N}}\right)\frac{d^n}{dx^n}(x^{n+m})
=\frac{d^n}{dx^n}(x^{n+m}\ln x) - \frac{(n+m)!}{m!}\frac{d^n}{dx^n}\left(\int{x^m \ln x \,dx^n}\right)+\frac{d^n}{dx^n}(C_n).
\end{equation*}
We know the following: 
\begin{equation*}
\frac{d^n}{dx^n}(C_n)=0, \,\,\,\,
\frac{d^n}{dx^n}(x^{n+m})=\frac{(n+m)!}{m!}x^m, \,\,\,\,
\frac{d^n}{dx^n}\left(\int{x^m \ln x \,dx^n}\right)=x^m \ln x, 
\end{equation*}
and, from Lemma \ref{l 4.1}, 
\begin{equation*}
\frac{d^n}{dx^n}(x^{n+m}\ln x)
=\frac{(n+m)!}{m!}x^m \ln x +n!x^m\sum_{k=1}^{n}{\frac{(-1)^{k+1}}{k}\binom{m+n}{m+k}}
.\end{equation*}
Hence, substituting back, we get 
\begin{equation*}
\begin{split}
\left(\sum_{N=1+m}^{n+m}{\frac{1}{N}}\right)\frac{(n+m)!}{m!}x^m
&=\frac{(n+m)!}{m!}x^m \ln x +n!x^m\sum_{k=1}^{n}{\frac{(-1)^{k+1}}{k}\binom{m+n}{m+k}} - \frac{(n+m)!}{m!}x^m \ln x \\
&=n!x^m\sum_{k=1}^{n}{\frac{(-1)^{k+1}}{k}\binom{m+n}{m+k}}
.\end{split}
\end{equation*}
Dividing both sides by $n!x^m$, we get the theorem.
\end{proof}
\begin{corollary} \label{c 4.1}
For $m=0$, Theorem \ref{t 4.1} becomes 
\begin{equation*}
\sum_{k=1}^{n}{\frac{1}{k}}
=\sum_{k=1}^{n}{\frac{(-1)^{k+1}}{k}\binom{n}{k}}
.\end{equation*}
\end{corollary}
\subsection{Applications of the alternating expression}
\subsubsection{Expression of the repeated harmonic sum in terms of an alternating sum}
In \cite{RepeatedSums}, the author showed how the repeated harmonic sum can be expressed in terms of a simple shifted harmonic sum. In this section, we express the repeated harmonic sum as an alternating sum. 
\begin{theorem} \label{t 4.2}
For any $n\in\mathbb{N^*}$ and any $m\in\mathbb{N}$, we have that 
\begin{equation*}
\sum_{k_{m+1}=1}^{n}{\cdots\sum_{k_2=1}^{k_3}{\sum_{k_1=1}^{k_2}{\frac{1}{k_1}}}}
=\sum_{k=1}^{n}{\frac{(-1)^{k+1}}{k}\binom{m+n}{m+k}}
.\end{equation*}
\end{theorem}
\begin{proof}
In \cite{RepeatedSums}, the following formula was proven, 
\begin{equation*}
\sum_{k_{m+1}=1}^{n}{\cdots \sum_{k_1=1}^{k_2}{\frac{1}{k_1}}}
=\binom{n+m}{m}\sum_{i=1+m}^{n+m}\frac{1}{i}
.\end{equation*}
By applying Theorem \ref{c 4.1}, we obtain the theorem. 
\end{proof}
\subsubsection{Expression of the repeated binomial-harmonic sum in terms of an alternating sum}
Similarly, in \cite{RepeatedSums}, the author also showed that the binomial-harmonic sum can be expressed as a shifted harmonic sum. In this section, we develop a new expression in terms of an alternating sum. 
\begin{theorem} \label{t 4.3}
For any $n,k\in\mathbb{N^*}$ and any $m\in\mathbb{N}$, we have that 
\begin{equation*}
\sum_{N_{k}=1}^{n}{\cdots\sum_{N_1=1}^{N_2}{\left[\binom{N_1+m}{m}\sum_{i=1+m}^{N_1+m}{\frac{1}{i}}\right]}}
=\sum_{i=1}^{n}{\frac{(-1)^{i+1}}{i}\binom{m+k+n}{m+k+i}}
.\end{equation*}
\end{theorem} 
\begin{proof}
In \cite{RepeatedSums}, the following formula was proven, 
\begin{equation*}
\sum_{N_k=1}^{n}{\cdots\sum_{N_1=1}^{N_2}{\left[\binom{N_1+m}{m}\sum_{i=1+m}^{N_1+m}{\frac{1}{i}}\right]}}
=\binom{n+m+k}{m+k}\sum_{i=1+m+k}^{n+m+k}{\frac{1}{i}}
.\end{equation*}
Applying Theorem \ref{t 4.1}, we obtain the theorem. 
\end{proof}
\subsubsection{Application to the expression of the $n$-th integral of $x^m \ln x$}
In this section, we present an alternative expression for the $n$-th integral of $x^m \ln x$. 
\begin{theorem} \label{t 4.4}
For any $n\in\mathbb{N^*}$ and any $m\in\mathbb{N}$, we have that 
\begin{equation*}
\int{x^m \ln x\,dx^n}
=\frac{x^{n+m}}{n!\binom{n+m}{m}} \ln x
-\frac{x^{n+m}}{n!{\binom{n+m}{m}}^2}\sum_{k=1}^{n}{\frac{(-1)^{k+1}}{k}\binom{m+n}{m+k}}+C_n
.\end{equation*}
\end{theorem}
\begin{proof}
From Corollary \ref{l 3.2}, we have 
\begin{equation*}
\int {x^m \ln x \,d x^n}
=\frac{x^{n+m}}{n!\binom{n+m}{m}} \left[ \ln x - \sum_{k=1+m}^{n+m}{\frac{1}{k}} \right]
+C_{n}
.\end{equation*}  
Hence, applying Theorem \ref{t 4.1}, we obtain the theorem. 
\end{proof}
\begin{corollary} \label{c 4.3}
For $m=0$, Theorem \ref{t 4.4} becomes 
\begin{equation*}
\int{\ln x\,dx^n}
=\frac{x^{n}}{n!} \ln x
-\frac{x^{n}}{n!}\sum_{k=1}^{n}{\frac{(-1)^{k+1}}{k}\binom{n}{k}}+C_n
.\end{equation*}
\end{corollary}
\subsubsection{Divergence of the alternating series}
The harmonic series was proven to diverge by several mathematicians including Nicole Oresme \cite{Oresme}, Pietro Mengoli \cite{Mengoli}, Johann Bernoulli \cite{JohannBernoulli}, and Jacob Bernoulli \cite{JacobBernoulli1,JacobBernoulli2}. Twenty additional proofs can also be found in \cite{kifowit2006harmonic}. In this section, we use this fact to proof the divergence of the alternating series presented in this section. 
\begin{theorem} \label{t 4.5}
For any $m\in\mathbb{N}$, we have that 
\begin{equation*}
\lim_{n \to \infty}{\left[\sum_{k=1}^{n}{\frac{(-1)^{k+1}}{k}\binom{m+n}{m+k}}\right]}=\infty
.\end{equation*}
\end{theorem}
\begin{proof}
\begin{equation*}
\lim_{n \to \infty}{\binom{n+m}{m}}
=\lim_{n \to \infty}{\frac{(n+m)!}{m!n!}}
=\frac{1}{m!}\lim_{n \to \infty}{(n+m)\cdots(n+1)}
=\infty
.\end{equation*}
\begin{equation*}
\lim_{n \to \infty}{\sum_{k=1+m}^{n+m}{\frac{1}{k}}}
=\lim_{n \to \infty}{\left[\sum_{k=1}^{n+m}{\frac{1}{k}}-\sum_{k=1}^{m}{\frac{1}{k}}\right]}
=\sum_{k=1}^{\infty}{\frac{1}{k}}-\sum_{k=1}^{m}{\frac{1}{k}}
=\infty
.\end{equation*}
Hence, by using Theorem \ref{t 4.1}, we prove that this alternating series diverges. 
\end{proof}
\bibliographystyle{elsarticle-num}
\bibliography{Repeated_Integration_v2}

\end{document}